\documentclass{amsart}

\newtheorem{theorem}{Theorem}[section]
\newtheorem{lemma}[theorem]{Lemma}
\newtheorem{proposition}[theorem]{Proposition}
\theoremstyle{definition}
\newtheorem{definition}[theorem]{Definition}
\newtheorem{example}[theorem]{Example}

\theoremstyle{remark}
\newtheorem{remark}[theorem]{Remark}

\numberwithin{equation}{section}

\newtheorem*{thm*}{Theorem}
\newtheorem*{notation}{Notation}
\newtheorem*{ack}{Acknowledgements}

\usepackage{graphicx}
\usepackage{amsmath,bm}
\usepackage{amssymb}
\usepackage[abbrev]{amsrefs}
\usepackage{mathtools}
\usepackage{enumerate}
\usepackage{esint}
\usepackage{color}

\usepackage{overpic}
\usepackage{mleftright}\mleftright

\newcommand{\bone}{\mathbf{1}}

\newcommand{\N}{\mathbb{N}}
\newcommand{\R}{\mathbb{R}}
\newcommand{\cK}{\mathcal{K}}

\newcommand{\dl}{\delta}

\newcommand{\eps}{\varepsilon}

\newcommand{\cl}[1]{\mkern 1.5mu\overline{\mkern-1.5mu#1\mkern-1.5mu}\mkern 1.5mu}

\newcommand{\relmiddle}[1]{\mathrel{}\middle#1\mathrel{}}

\numberwithin{equation}{section}
\allowdisplaybreaks[3]

\begin{document}

\title[Discrete approximation of reflected Brownian motions]{Discrete approximation of reflected Brownian motions by Markov chains on partitions of domains}



\author{ Masanori Hino}
\address{Department of Mathematics, Kyoto University, Kyoto 606-8502,  Japan}
\curraddr{}
\email{hino@math.kyoto-u.ac.jp}
\thanks{The first author was supported by JSPS KAKENHI Grant Number 19H00643.}

\author{Arata Maki }
\address{ Ichinomiya-shi, Aichi, Japan}
\curraddr{}
\email{hamdam@ezweb.ne.jp}
\thanks{}

\author{Kouhei Matsuura}
\address{Institute of Mathematics, University of Tsukuba, 1-1-1, Tennodai, Tsukuba, Ibaraki, 305-8571, Japan}
\curraddr{}
\email{kmatsuura@math.tsukuba.ac.jp}
\thanks{The third author was supported by JSPS KAKENHI Grant Number 20K22299  and 	22K13926.
}


\date{}

\dedicatory{}


\begin{abstract}
In this paper,  we study discrete approximation of reflected Brownian motions on domains in Euclidean space.  Our approximation is given by a sequence of Markov chains on partitions of the domain, where we allow uneven or random partitions. We provide sufficient conditions for the weak convergence of the Markov chains.
\end{abstract}

\maketitle

\section{Introduction}

Discrete approximation of a space is a problem that arises in various fields of mathematics.  In probability theory,  this is often studied in connection with the scaling limits of random walks or Markov chains.  A typical example is approximating the Brownian motion on the $d$-dimensional Euclidean space $\R^d$  by means of simple random walks on the scaled lattices $n^{-1}\mathbb{Z}^d=\{x \in \R^d \mid nx \in \mathbb{Z}^d \}$,   $n \in \mathbb{N}$,  and this result has been extended in various directions.  For example,  Stroock and Zheng~\cite{SZ} proved that a symmetric diffusion process on $\R^d$ corresponding to a uniformly elliptic divergence form operator can be approximated by a sequence of Markov chains on $n^{-1}\mathbb{Z}^d$,   $n \in \mathbb{N}$.  Recently,  discrete approximation for  a class of jump processes on $\R^d$ has also been considered; see,  e.g.,  \cite{BKU, CKK}.  

When the state space under consideration is the closure $\cl{D}$ of a proper domain $D$ of $\R^d$, a canonical diffusion process thereon is the (normally) reflected Brownian motion (RBM in abbreviation).   
  Roughly speaking,  the RBM  on $\cl{D}$ is a continuous Markov process taking values in $\cl{D}$ that behaves like the Brownian motion in $D$ and is instantaneously pushed back along the inward normal direction when it hits the boundary $\partial D$ of $D$.  If $\partial D$ is sufficiently smooth (e.g., if $D$ is a $C^2$-domain), then the RBM can be constructed   as a solution to the deterministic Skorohod problem or the submartingale problem.   In both cases, the RBM solves a stochastic differential equation described by  \eqref{eq:skorohod} below,  which also shows that the associated generator is the Neumann Laplacian on $\cl{D}$.  In \cite{SV},  Stroock and Varadhan studied discrete approximation  to prove an invariance principle for a class of diffusion processes  including the RBM on a $C^2$-domain; see \cite[Theorem~6.3]{SV} for details.  Burdzy and Chen~\cite{BC2} extended this result in view of the smoothness of $\partial D$.   More precisely,   
they showed  that RBM on a general bounded domain can be approximated by simple random walks on subsets of $(2^{-n}\mathbb{Z}^d) \cap D$,  $n \in \mathbb{N}$.  This RBM is constructed by a Dirichlet form; this is identified with the RBM constructed by the methods stated above in the case of smooth domains (see,  e.g.,  \cite[Theorem~4.4]{C}).

On the other hand,  discrete approximation of a Euclidean domain by inhomogeneous point sets has often been studied from an analytic perspective.  For example,  a recent study by Davis and Sethuraman~\cite{DS}  discussed approximation of paths in a domain $D$ that minimize cost functions  by discrete paths on graphs consisting of independent  and identically distributed random points in $D$; for the precise statement, see  \cite[Theorems~2.1 and 2.2]{DS}.  We also note that Garc\'{\i}a Trillos~\cite{Gar} recently studied a similar discrete approximation of the Wasserstein distance on the $d$-dimensional torus.  In order to gain a better understanding of discrete approximation by such point sets,  a natural problem to consider next is the scaling limits of Markov chains thereon (see,  e.g., \cite{cr,cr2} for related studies). Keeping such motivation in mind,  in this paper,  we study approximation to the RBM by  Markov chains on a sequence of general partitions  $\{\cK^{(n)}\}_{n=1}^\infty$ of $\cl{D}$ whose sizes converge to zero appropriately as $n \to \infty$.   A typical example of such  partitions  is  the Voronoi partition  associated 
with uniformly distributed independent random points in $D$ (see Example~\ref{ex:voronoi} below).  In such a situation,  it is a non-trivial problem to construct appropriate Markov chains approximating the RBM because  of the inhomogeneity of the partitions.  Our main purpose herein is to address this difficulty, and in Theorem~\ref{thm:4} (our main result), 
we construct a specific Markov chain $X^{(n)}$ on each $\cK^{(n)}$ in such a way that the inhomogeneity is taken into consideration,  and we prove that its distribution converges weakly to the RBM in the Skorohod space as $n \to \infty.$ 

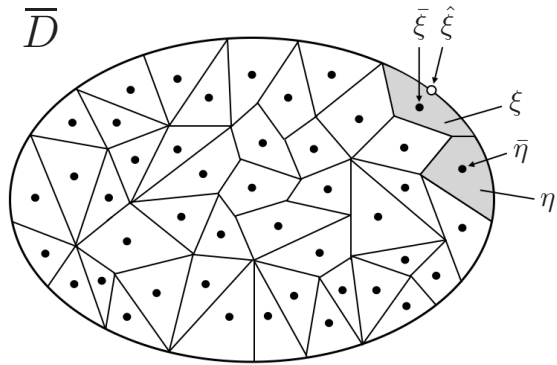
\begin{figure}[t]
\centering
    \vspace{0.30cm}
  \begin{overpic}[scale=0.28]{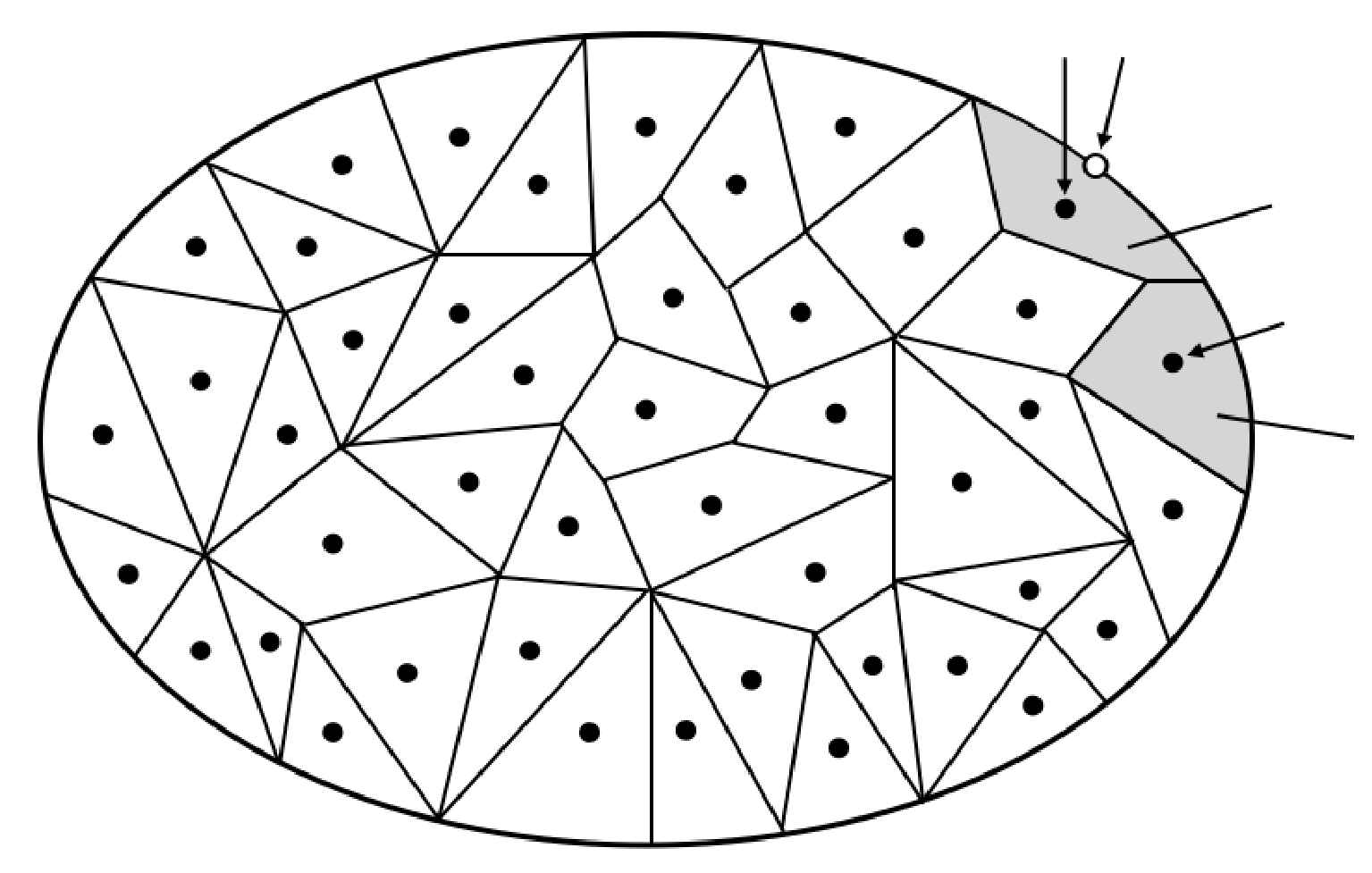}
  \put(5,59.5){\huge $\cl{D}$}
    \put(76.5,62){\large $\bar{\xi}$}
            \put(81.5,62){\large $\hat{\xi}$}
        \put(94,48){\large $\xi$}
                \put(94.8,39.5){\large $\bar{\eta}$}
        \put(99.7,30.5){\large $\eta$}
    \end{overpic}
    \caption{Figure of a partition of $\cl{D}$. The black dot on each cell denotes the center of gravity of that cell. The white dot denotes the nearest point of $\partial D$ to the center of gravity.}
    \label{figure:partition}
\vspace{-0.5cm}
\end{figure}
We begin by focusing on the case of $D=\R^d$ to explain how to define a desirable sequence of Markov chains  $\{X^{(n)}\}_{n=1}^\infty$ and scaling factors.   Let $n \in \N$ and let $\xi $ be an element of $\cK^{(n)}$,  which we call a {\it cell}.  The set $\mathcal{N}_{\xi}^{(n)}$ of all neighbor cells of $\xi$ is defined as a collection of cells that lie within appropriate distance from $\xi$. More precisely, we  take a function $\rho_n\colon \cK^{(n)} \to (0,\infty)$ and  define 
\begin{equation}\label{eq:nn}
\mathcal{N}_{\xi}^{(n)}=\{\eta \in \cK^{(n)}  \mid |\bar{\eta}-\bar{\xi}| <\rho_n(\xi)\},
\end{equation}
where $\bar{\eta} \in \R^d$ denotes the center of gravity of $\eta$ (see Fig.~\ref{figure:partition}).   A natural Markov chain $M^{(n)}$ on $\cK^{(n)}$ is associated with a generator $U^{(n)}\colon \R^{\cK^{(n)}} \to \R^{\cK^{(n)}}$  defined by
 \[
U^{(n)}f(\xi)=\frac{1}{\rho_n(\xi)^2}\sum_{\eta \in \mathcal{N}_{\xi}^{(n)}}\left(f(\eta)-f(\xi) \right)\frac{m(\eta)}{m(O_{\xi}^{(n)})},
\]
where $O_{\xi}^{(n)}=\bigcup_{\eta \in \mathcal{N}^{(n)}_{\xi}} \eta$ and $m$ is the $d$-dimensional Lebesgue measure.   This process behaves as follows:  when $M^{(n)}$ visits a cell $\xi$,   it stays there for an exponentially distributed waiting time with parameter $1/\rho_n(\xi)^2$,  and then it jumps to $\eta \in \mathcal{N}^{(n)}_{\xi}$ with probability $m(\eta)/m(O_\xi^{(n)})$.  
In order to discuss the weak convergence of $M^{(n)}$,   we assume $\sup_{\xi \in \cK^{(n)}}\rho_n(\xi)<\infty$ so that $M^{(n)}$ is conservative,  and we consider the convergence of $U^{(n)}$ under a projection map $\pi_n \colon C(\R^d) \to \R^{\cK^{(n)}}$ defined by $
\pi_n f(\xi)=f(\bar{\xi})$.  
 Let $f\colon \R^d \to \R$ be a smooth function on $\R^d$.  Then, for any $\xi \in \cK^{(n)}$, we have
\begin{align*}
U^{(n)}(\pi_n f)(\xi)=\frac{1}{\rho_n(\xi)^2}\sum_{\eta \in \mathcal{N}_{\xi}^{(n)}}\left(f(\cl{\eta})-f(\cl{\xi}) \right)\frac{m(\eta)}{m(O_{\xi}^{(n)})}. 
\end{align*}
 By applying Taylor's formula to the right-hand side of the above equation,  it is expanded as
\begin{align}\label{eq:generator0}
&\frac{1}{\rho_n(\xi)^2} \sum_{\eta \in \mathcal{N}_{\xi}^{(n)}}  \langle \nabla f(\cl{\xi}),\bar{\eta} -\cl{\xi}\rangle   \frac{ m(\eta)}{m(O_{\xi}^{(n)})}  \\
&+\frac{1}{\rho_n(\xi)^2}  \sum_{\eta \in \mathcal{N}_{\xi}^{(n)}}\left( \int_{0}^1 (1-t)\langle \bar{\eta} -\cl{\xi}, \nabla^2 f(\cl{\xi}+t(\bar{\eta} -\cl{\xi}))(\bar{\eta} -\cl{\xi})\rangle\,dt \right) \frac{m(\eta)  }{m(O_{\xi}^{(n)})},  \notag
\end{align}
where $\langle \cdot, \cdot \rangle$ denotes the standard inner product on $\R^d.$ When $\cK^{(n)}$ has no  lattice structure,  we note that the first term in \eqref{eq:generator0} does not necessarily vanish.   Some computations imply that this is bounded above by 
$ |\nabla f(\bar{\xi})|  \eps_n(\xi)/\rho_n(\xi)^2$ up to a multiplicative constant depending only on $d$ (see \eqref{eq:b} and Lemma~\ref{lem:A} below).  Here,  $\eps_n(\xi)$ is akin to the maximal diameter among all neighbor cells of $\xi$ (see \eqref{eq:N2} below for the definition). From this estimate,  as $n \to \infty$,  $\overline{M^{(n)}}$ and $M^{(n)}$ are expected to converge weakly to a (possibly time-changed) $d$-dimensional Brownian motion as long as $\lim_{n \to \infty}\sup_{\xi \in \cK^{(n)}} \eps_n(\xi)/\rho_n(\xi)^2=0$ and $\lim_{n \to \infty}\sup_{\xi \in \cK^{(n)}} \rho_n(\xi)=0$.   In fact,  by introducing a correction term to cancel the first term in \eqref{eq:generator0},  we obtain a convergence result under milder conditions.  This is how  we define the  Markov chain $X^{(n)}$ corresponding to a generator that  suitably modifies  $U^{(n)}$.


 If $ D$ is a proper domain of $\R^d$,   we need to 
 take reflections at the boundary into consideration. For this purpose, we classify the cells into two types, 
 ``interior cells'' and ``boundary cells.'' If $\xi \in \cK^{(n)}$ is an interior cell,   the set $\mathcal{N}_{\xi}^{(n)}$ of all neighbor cells of $\xi$ is defined as in \eqref{eq:nn}.   If $\xi$ is a boundary cell,   we take $\hat{\xi} \in \partial D$ satisfying  $|\hat{\xi}-\bar{\xi}|=\text{dist}(\bar{\xi},\partial D)$   (see the white dot in Fig.~\ref{figure:partition}) and define $\mathcal{N}_{\xi}^{(n)}=\{\eta \in \cK^{(n)} \mid |\bar{\eta}-\hat{\xi}|<\rho(\xi)\}$.  This definition is made so that the term 
 \[ \frac{1}{\rho_n(\xi)^2} \sum_{\eta \in \mathcal{N}_{\xi}^{(n)}}  \langle \nabla f(\hat{\xi}),\bar{\eta} -\hat{\xi}\rangle \, \frac{m(\eta)}{m(O_{\xi}^{(n)})}\] appears instead of the first term in \eqref{eq:generator0}.  
Writing $\nu(\hat{\xi})$ for the inward unit normal vector at $\hat{\xi}$,  we expect this new term to approximate a constant multiple of $\langle \nabla f(\hat{\xi}),\nu(\hat{\xi})\rangle$,  which should appear for the generator  to converge to the Neumann Laplacian on $\cl{D}$.  However,  justification of this also requires a correction term and a sufficient smoothness of $\partial D$.   For this,  we assume that $D$ is a $C^{1,\alpha}$-domain for some $\alpha \in (0,1].$

We emphasize that it is not clear whether the operators with correction terms become generators.  We need some quantitative estimates for the  correction terms to provide modified generators.  After these preparations,  we can derive a sufficient condition under which these  generators converge to the standard Laplacian on the subspace of $C^2(\R^d)$ with Neumann boundary condition on $\partial D.$  
Because the Neumann Laplacian on $\cl{D}$ is the generator of RBM,  we can also strengthen this convergence result to the weak convergence of $\{X^{(n)}\}_{n=1}^\infty$ to RBM.

We make some comments on related studies.  A recent study by Ishiwata and Kawabi~\cite{IK} discussed a discrete approximation of a drifted Schr\"{o}dinger operator on a complete Riemannian manifold by Markov chains on its (possibly random) partitions.  
The Ricci curvature is heavily involved in this approximation.  Though the motivation of the study and constraints of the scaling factors are different from ours,  their study suggests that our setting could be more generalized.  As already mentioned above,  Stroock and Varadhan~\cite{SV} used Markov chains  to prove discrete approximation to diffusion processes with boundary conditions. Although the state spaces of the Markov chains are general,  no specific constructions for correction terms are done, unlike in our study.  In addition, their proof makes use of the well-posedness of submartingale problems, whereas the key point of our proof is to obtain discrete approximation to the Neumann Laplacian with respect to the uniform convergence topology. 

The rest of this paper is organized as follows. In Section~2, we set up
our framework, state main theorems (Theorems~\ref{thm:1},  \ref{thm:2},  \ref{thm:3},  and \ref{thm:4}),  and present
some typical examples.  In Section~3, we give some preliminary results to define and estimate the correctors.   Section~4 is devoted to proving Theorems~\ref{thm:1} and \ref{thm:2}, with
Theorem~\ref{thm:3} obtained by combining these two theorems.  
In Section~5,  we apply  Theorem~\ref{thm:3} to  prove Theorem~\ref{thm:4}.
In Section~6,  we discuss a few problems as future issues. Appendix~A provides some boundary estimates.  Appendix~B provides a proof of Proposition~\ref{prop:core} about a claim on a core for the Neumann Laplacian.

\begin{notation}
The following notation is used in this paper.
\begin{itemize}
\item  For $x \in \R^d$ and  $i \in \N$  with $1\le i\le d$,   we write $x_i$ for the $i$-th coordinate of $x$.
\item  The closure and boundary of $S \subset \R^d $ in $\R^d$ are denoted by $\cl{S} $ and $\partial S$,  respectively.
\item $\langle\cdot,\cdot\rangle$ and $|\cdot|$ denote the standard inner product and norm of $\R^d$, respectively. 
\item For $S \subset \R^d$ with $S\ne\emptyset$, we set $\text{diam}(S)=\sup_{x,y \in S}|x-y|$. 
\item For $x\in\R^d$ and $r\in (0,\infty)$, $B(x,r)$ (resp.\ $\cl{B}(x,r)$) denotes the open (resp.\  closed) ball in $\R^d$ with center $x$ and radius $r$.
\item $m$ denotes the Lebesgue measure on $\R^d$.  
\item  We write $\omega_d=m(B(0,1))$ and set $\omega_0=1$.
\item For a Borel subset $S \subset \R^d$ with $m(S) \in (0,\infty)$ and an $m$-integrable function $f\colon S \to [-\infty,\infty]$, we define \[\fint_{S}f(x)\,m(dx)=\frac{1}{m(S)}\int_{S}f(x)\,m(dx).\] If $f$ is a vector valued function,  $\fint_{S}f(x)\,m(dx)$ is defined similarly if it exists.
\item For $d,k \in \N$,  $\R^d \otimes \R^k$ denotes the set of all $d\times k$ real matrices.  For $A \in \R^d \otimes \R^k$,  ${}^{\,t\!}A$ denotes its transposed matrix. 
\item For $A \in \R^d\otimes \R^d$,  we write $\|A\|$ and $\text{Tr}[A]$ for the operator norm and the trace, respectively.  We write $I_d$ for the identity matrix of order $d$.
\item For $a,b \in \R$, we set $a \vee b =\max \{a,b\}$ and $a \wedge b=\min \{a,b\}.$
\item $\# I$ is the cardinality (number of elements) of a set $I$.
\item $\inf\emptyset=\infty$ by convention.
\end{itemize}
\end{notation}

\section{Statement of results}

Let $D$ be a (possibly unbounded) domain in $\R^d$ with $m(\partial D)=0$.   First, we introduce the following definition.  

\begin{definition}
A countable collection $\mathcal{K}$ of subsets in $\R^d$ is said to be a {\it partition} of $\overline{D}$ if the following conditions are satisfied.
\begin{enumerate}
\item  Any $\xi \in \mathcal{K}$ is a compact subset of $\cl{D}$ with $m(\xi)>0$.
\item For any $\xi,\eta \in \mathcal{K}$ with $\xi \neq \eta$, we have $m(\xi \cap \eta)=0$.
\item We have $\cl{D}=\bigcup_{\xi \in \mathcal{K}}\xi $.
\item  For any $x \in \cl{D}$ and $r \in (0,\infty)$, we have
\begin{align}\label{eq:lf}
\#\{\xi \in \mathcal{K} \mid \xi \cap B(x,r) \neq \emptyset \}<\infty.
\end{align}
\end{enumerate} 
\end{definition}

For the moment,  we fix a partition $\mathcal{K}$ of $\cl{D}$.   For $\xi \in \mathcal{K}$,  let $\bar{\xi} \in \R^d$ denote  the center of gravity of $\xi$,  that is,
\[
\bar{\xi}=\fint_{\xi}x\,m(dx).
\]
Fix a positive function $\delta \colon \mathcal{K} \to (0,\infty)$ and set
\begin{equation}\label{eq:boundarycell}
\partial \mathcal{K}=\{\xi \in \mathcal{K} \mid B(\bar{\xi},\dl(\xi)) \not \subset D\}.
\end{equation}
For any $\xi \in \partial \mathcal{K}$,  there exists (possibly more than one point) $x \in \partial D$ such that $|\bar{\xi}-x|=\text{dist}(\bar{\xi},\partial D)$.  We choose and fix this $x$ and denote it by $\hat{\xi}$. 
For $\xi \in \cK$,  define $\check{\xi} \in \R^d$ by
\begin{equation*}
\check{\xi}=\begin{cases}
\bar{\xi} &\quad \text{if}\quad \xi \in \cK \setminus \partial \mathcal{K}, \\
\hat{\xi} &\quad \text{if} \quad \xi \in \partial \mathcal{K}.
\end{cases}
\end{equation*}
We take $\rho \colon \cK \to (0,\infty)$ such that 
\begin{align}\label{eq:r}
\rho(\xi)&
\left\{
    \begin{alignedat}{2}
&=\dl(\xi)  &\quad&   \text{ if }\quad \xi \in \cK \setminus \partial \cK ,  \\
&> \dl(\xi) &  \quad   &    \text{ if }\quad \xi \in \partial \cK.
 \end{alignedat}
    \right.
    \end{align}
For $\xi \in \mathcal{K} $,  define 
\begin{align}\label{eq:N1}
\begin{split}
\mathcal{N}_{\xi} &=\{\eta \in \mathcal{K} \mid \bar{\eta} \in B(\check{\xi},\rho(\xi))\},\quad
O_{\xi}=\bigcup_{\eta \in \mathcal{N}_{\xi} }\eta.
\end{split}
\end{align}
We see that \eqref{eq:r} implies $\mathcal{N}_\xi \neq \emptyset$. We also define 
\begin{equation}\label{eq:N2}
\widetilde{\mathcal{N}}_{\xi}=\{\eta \in \mathcal{K} \mid \eta \cap  B(\check{\xi},\rho(\xi)) \neq \emptyset\},\quad \eps(\xi)=\max_{\eta \in \mathcal{N}_\xi \cup \widetilde{N}_\xi} \sup_{y \in \eta}|y-\bar{\eta}|.
\end{equation}
Define an operator  $L\colon \R^{\cK} \to \R^{\cK}$ as
 \begin{equation}\label{eq:gener}
Lf(\xi)=\frac{1}{q(\xi)}\sum_{\eta \in \mathcal{N}_{\xi}}\left(f(\eta)-f(\xi) \right)\left(1-c(\xi,\eta)\right)\frac{m(\eta)}{m(O_{\xi})},
\end{equation}
where $q(\xi)$ and $c(\xi,\eta)$ will now be explained.

Let $\mathcal{N}_{\xi}$ be expressed as  \[\mathcal{N}_{\xi}=\{\eta^{(1)},\ldots,\eta^{(N_\xi)}\},\]
where $N_\xi=\# \mathcal{N}_\xi$ and $1 \le N_\xi <\infty$ by \eqref{eq:lf}.  We use this expression to define  $A(\xi) \in \R^d \otimes \R^{N_\xi}$ whose $(i,j)$-th element is given by
\begin{align}\label{eq:A1}
        &(\overline{\eta^{(j)}}_i-\check{\xi}_i) \frac{m(\eta^{(j)})}{m(O_{\xi})}.
\end{align}
Next,  we assume that $D$ is a $C^1$-domain and  let $\nu$ denote the inner unit normal vector on $\partial D$.  
Let $\mathrm{ B}(\cdot,\cdot)$ denote the beta function and set\footnote{$\beta_d$ coincides with $\frac{2\omega_{d-1}}{(d+1)\omega_d}$ (see \eqref{eq:uph1} in Appendix~A). }
 \[\beta_d=\frac{2}{(d+1)\mathrm{ B}(\frac{1}{2},\frac{d+1}{2})}.\]
Define $b \colon \cK \to \R^d$ by
\begin{align}\label{eq:b}
b(\xi)&=
\left\{
    \begin{alignedat}{2}
&  \sum_{\eta \in \mathcal{N}_{\xi}} (\bar{\eta}-\bar{\xi}) \frac{m(\eta)}{m(O_{\xi})} &\quad&   \text{ if }\quad \xi \in \cK \setminus \partial \cK ,  \\
&\sum_{\eta \in \mathcal{N}_{\xi}} (\bar{\eta}-\hat{\xi}) \frac{m(\eta)}{m(O_{\xi})}  - \beta_d \rho(\xi)\nu(\hat{\xi})&  \quad   &    \text{ if }\quad \xi \in \partial \cK.
 \end{alignedat}
    \right.
    \end{align}

For $\xi \in \cK$,  let $A^{+}(\xi) \in \R^{N_\xi}\otimes \R^d $ denote the generalized inverse matrix\footnote{For $l,m \in \N$ and $A \in \R^{l}\otimes \R^{m}$,  the generalized inverse matrix $A^{+} \in \R^{m}\otimes \R^{l}$ of $A$
is defined as a unique matrix satisfying the following four equations:
\begin{align*}
AA^{+}A=A,\quad A^{+}AA^{+}=A^{+},\quad {}^{\,t\!}(AA^{+})=AA^{+},\quad \text{and}\quad {}^{\,t\!}(A^{+}A)=A^{+}A.
\end{align*}} of $A(\xi)$   and define $c(\xi) \in \R^{N_\xi}$ by 
\begin{equation}\label{eq:c0}
c(\xi)=A^{+}(\xi)b(\xi).
\end{equation}
We will often write $c(\xi)={}^t(c(\xi,\eta^{(1)}),\ldots,c(\xi,\eta^{(N_\xi)})).$ Define $Q\colon \cK \to \R^{d}\otimes \R^d$ and $q\colon \cK \to \R$ by  \begin{equation}\label{eq:dQ1}
Q(\xi)= \sum_{\eta \in \mathcal{N}_{\xi}} (\bar{\eta}-\check{\xi})\otimes (\bar{\eta}-\check{\xi})\, \frac{m(\eta)}{m(O_{\xi})},\quad q(\xi)=\frac{1}{d}\text{Tr}[Q(\xi)].
\end{equation}
The following lemma is used to define a generator on $\cK$ and gives its quantitative estimate.  Define $a_1\colon (0,\infty) \to \R$ by
\[
a_1(t)=\frac{1}{d+2}\frac{(1-t)^{d+2}}{ (1+t)^d}-t^2
\]
and fix $c_1 \in (0,\min\{t>0 \mid a_1(t)=0\})$.
\begin{lemma}\label{lem:A}
Assume 
\begin{equation}\label{eq:A}
\sup_{\xi \in \cK \setminus \partial \cK}\frac{\eps(\xi)}{\rho(\xi)} \le c_1.
\end{equation}
Then,  there exists $C>0$ depending only on $d$ and $c_1$ such that for any $\xi \in \cK \setminus \partial \cK$,
\begin{equation*}
\max_{\eta \in \mathcal{N}_{\xi}}|c(\xi,\eta)|  \le  \frac{C \eps(\xi)}{\rho(\xi)} \quad  \text{and }\quad q(\xi) \ge C^{-1}\rho(\xi)^2.
\end{equation*} 
\end{lemma}

The proof of Lemma \ref{lem:A} will be given in the next section.   In order to treat the case of $\xi \in  \partial \cK$,  we need to assume that $D$ is a $C^{1,\alpha}$-domain,  the definition of which is as follows. 
\begin{definition}\label{defn:Dc1a}
Let $\alpha \in (0,1].$ We say that $D$ is a $C^{1,\alpha}$-domain if  there exist $C>0$ and $R>0$ such that for any $x \in \partial D$,   there exist a $C^{1,\alpha}$-function $F_x \colon \R^{d-1} \to \R$ satisfying $F_x(0)=0$, $\nabla F_x(0)=0$,  $\sup_{y' \in \R^{d-1}}|\nabla F_x(y')|\le C$,  \[|\nabla F_x (y')-\nabla F_x (z')|\le C|y'-z'|^\alpha ,\quad y',z' \in \R^{d-1},\] and an orthonormal coordinate system $CS_x$ with its origin at $x$ such that
\begin{equation}\label{eq:Dcs}
B(x,R) \cap D=\{y=(y',y_d) \in B(0,R) \text{ in $CS_x$} \mid y_d>F_x(y')\},
\end{equation}
where $y'=(y_1,\ldots,y_{d-1})$.
\end{definition}

If $D$ is a $C^{1,\alpha}$-domain for some $\alpha \in (0,1]$,  it satisfies the uniform cone condition (see,  e.g., \cite[Chapter~4]{AF}). Thus, there exists $\kappa \in (0,\pi]$ with the following property: for any $\xi \in \partial \cK$ there exists a (closed) cone in $\cl{D}$ of aperture angle $\kappa$ with vertex at $\hat{\xi}$.  See \cite[Section~4.4]{AF} for definitions of these terminologies. We fix such a positive constant $\kappa$.  Let $c \in (0,1)$ and assume $\sup_{\xi \in \partial \cK}\eps(\xi)/\rho(\xi) \le c$. Let $\xi \in \partial \cK$. Lemma~\ref{lem:edd} below implies that $\cl{D} \cap B(\hat{\xi},(1-c)\rho(\xi)) \subset O_\xi$. Thus, there exists a cone $C$ of height $(1-c)\rho(\xi)$, aperture angle $\kappa$ with vertex at $\hat{\xi}$ such that $C \subset O_\xi$. On account of this fact, we fix a positive number $R_D$ independent of $\xi$ such that 
\[
R_D >\inf \left\{r>0 \mid \text{$B(e,\rho(\xi)/r) \subset O_\xi$ for some $e \in O_\xi$}\right\}.
\]
Define $a_2 \colon (0,\infty) \to \R$ by 
\[
a_2(t)=\frac{ R_{D}^{-(d+2)}}{d+2}\frac{1}{(1+t)^d} -t^2 -2t
\]
and fix $c_2\in (0,c \wedge \min\{t>0 \mid a_2(t)=0\}).$

\begin{lemma}\label{lem:B}
Assume that $D$ is a $C^{1,\alpha}$-domain for some $\alpha \in (0,1]$.  Assume in addition that 
\begin{equation}\label{eq:B}
\sup_{\xi \in  \partial \cK}\frac{\eps(\xi)}{\rho(\xi)}\le c_2.
\end{equation}
Then,  there exists $C>0$ depending only on $D$ and $c_2$ such that for any $\xi \in \partial \cK$,
\[
\max_{\eta \in \mathcal{N}_{\xi}}|c(\xi,\eta)|  \le C\left(\frac{ \eps(\xi)}{\rho(\xi)}+\rho(\xi)^\alpha \right) \quad  \text{and }\quad q(\xi) \ge C^{-1}\rho(\xi)^2.
\]
\end{lemma}

The proof  of Lemma \ref{lem:B} is postponed until the next section.   Assume that $\sup_{\xi \in \cK}\eps(\xi)/\rho(\xi)$ and $\sup_{\xi \in \cK}\rho(\xi)$ are sufficiently small and $\inf_{\xi \in \cK}\rho(\xi)>0$.  Then,   Lemmas~\ref{lem:A} and \ref{lem:B} imply that
\begin{equation}\label{eq:cflg1}
\inf_{\xi \in \cK} q(\xi)>0\text{ and } |c(\xi,\eta)|<1 \text{ for any $\xi \in \cK$ and $\eta \in \mathcal{N}_\xi$.}
\end{equation}
Note that $\inf_{\xi \in \cK}\rho(\xi)>0$ holds automatically if $D$ is bounded.   We  endow $\cK$ with the discrete topology and let  $\mathcal{B}_b(\cK)$  denote the space of bounded Borel functions on $\cK$ equipped with the sup-norm.   When \eqref{eq:cflg1} holds,  the operator $L$ defined in \eqref{eq:gener} is a bounded linear operator on $\mathcal{B}_b(\cK)$.  Also, the coefficients $1-c(\xi,\eta)$ in $L$ are all positive.  Thus,  $(L,\mathcal{B}_b(\cK))$ generates a continuous-time Markov chain on $\cK$  whose semigroup is given by $\{e^{tL}\}_{t>0}$ (see,  e.g., \cite[Chapter~4.2]{EK}).

In the following,  we see  how this generator $L$ approximates the Laplace operator $\Delta$ on $\R^d$.
Define a map $\pi\colon C(\R^d) \to \R^{\cK}$  by
\[
\pi f(\xi)=f(\bar{\xi}).
\]
Here,  $ C(\R^d) $ denotes the space of all continuous  real-valued functions on $\R^d$.   We also denote by $C^2(\R^d)$ the space of all twice  continuously differentiable  real-valued  functions on $\R^d.$
\begin{theorem}\label{thm:1}
Assume that \eqref{eq:A} holds.  Then, there exists $C>0$ depending only on $d$ and $c_1$ such that   for any  $f \in C^2(\R^d)$ and  $\xi \in \cK \setminus \partial \cK$,
\begin{align*}
&\left| L(\pi f)(\xi)- \pi \left(\frac{\Delta f}{2}\right)(\xi) \right| \notag \\
&\le  \frac{C\eps(\xi)}{\rho(\xi)} \sup_{y \in B(\bar{\xi},\rho(\xi))}\|\nabla^2 f(y)\| +C \sup_{y \in B(\bar{\xi},\rho(\xi))} \|\nabla^2 f(y)-\nabla^2 f(\bar{\xi})\|.
\end{align*}
\end{theorem}
Define  a subspace $C_{\text{Neu}}^2(\R^d)$ of $C^2(\R^d)$ as
\[
C_{\text{Neu}}^2(\R^d)=\left\{f \in C^2(\mathbb{R}^d) \relmiddle| \langle \nabla f(x),\nu(x) \rangle=0\text{ for every $x \in \partial D$} \right\}.
\]
\begin{theorem}\label{thm:2}
Assume that $D$ is a $C^{1,\alpha}$-domain for some $\alpha \in (0,1]$.  Assume in addition  that \eqref{eq:B} holds.  Then,  there exists $C>0$ depending only on $D$ and $c_2$ such that for any $f \in C_{\text{\rm Neu}}^2(\R^d)$  and  $\xi \in \partial \cK$ with $\rho(\xi)<R$,
\begin{align*}
\left| L(\pi f)(\xi)- \pi \left(\frac{\Delta f}{2} \right)(\xi)  \right| &\le C\left(\frac{ \eps(\xi) \vee \dl(\xi)}{\rho(\xi)}+\rho(\xi)^\alpha \right) \sup_{y \in B(\hat{\xi},\rho(\xi))}\|\nabla^2 f(y)\| \\
&\quad+C \sup_{y \in B(\hat{\xi},\rho(\xi))} \|\nabla^2 f(y)-\nabla^2 f(\hat{\xi}) \|.
\end{align*} 
Here,  $R$ is the positive constant in Definition~\ref{defn:Dc1a}.
\end{theorem}

In what follows,  we assume that $D$ is a $C^{1,\alpha}$-domain for some $\alpha \in (0,1]$.  Let $\{\cK^{(n)}\}_{n=1}^\infty$ be a sequence of partitions  of $\cl{D}$.  Let $n \in \N$ and fix a positive function $\dl_n \colon \cK^{(n)} \to (0,\infty)$  to define
\[
\partial \mathcal{K}^{(n)}=\{\xi \in \mathcal{K}^{(n)} \mid B(\bar{\xi},\dl_n(\xi)) \not \subset D\}.
\]
 For each $\xi \in \partial \mathcal{K}^{(n)}$,  we  fix $\hat{\xi} \in \partial D$ such that 
$|\bar{\xi}-\hat{\xi} |=\text{dist}(\bar{\xi},\partial D).$  
 We also take $\rho_n \colon \cK^{(n)} \to (0,\infty)$ satisfying
\begin{align*}
\rho_n(\xi)&
\left\{
    \begin{alignedat}{2}
&=\dl_n(\xi)  &\quad&   \text{ if }\quad \xi \in \cK^{(n)} \setminus \partial \cK^{(n)} ,  \\
&> \dl_n(\xi) &  \quad   &    \text{ if }\quad \xi \in \partial \cK^{(n)}.
 \end{alignedat}
    \right.
    \end{align*}
Let $\xi \in \mathcal{K}^{(n)}$ and define $\mathcal{N}_{\xi}^{(n)}$,  $O_{\xi}^{(n)}$,  and $\eps_n(\xi)$ in the same way as \eqref{eq:N1} and \eqref{eq:N2}.   Let  $N_{\xi}^{(n)}=\# \mathcal{N}_{\xi}^{(n)}$ and define a $d \times N_{\xi}^{(n)}$ matrix $A^{(n)}(\xi)$
 in the same way as \eqref{eq:A1}.  We write $A^{(n),+}(\xi)$ for the generalized inverse matrix.   Define vectors $b^{(n)}(\xi) \in \R^d$,  $c^{(n)}(\xi) \in \R^{N_{\xi}^{(n)}}$,  a matrix $Q^{(n)}(\xi) \in \R^d \otimes \R^d$,  and $q^{(n)}(\xi) \in \R$ in the same way as \eqref{eq:b}, \eqref{eq:c0},  and \eqref{eq:dQ1}, respectively.  Define  an operator $L^{(n)} \colon \R^{\cK^{(n)}} \to \R^{\cK^{(n)}}$ by
\[
L^{(n)}f(\xi)=\frac{1}{q^{(n)}(\xi)}\sum_{\eta \in \mathcal{N}_{\xi}^{(n)}}(f(\eta)-f(\xi))(1-c^{(n)}(\xi,\eta))\frac{m(\eta)}{m(O_{\xi}^{(n)})}. \]
Assume that $\sup_{\xi \in \cK^{(n)}}\eps_n(\xi)/\rho_n(\xi)$ and $\sup_{\xi \in \cK^{(n)}}\rho_n(\xi)$ are sufficiently small and  $\inf_{\xi \in \cK^{(n)}}\rho_n(\xi)>0$.   Then,  
Lemmas~\ref{lem:A} and \ref{lem:B} imply that 
\begin{equation}\label{eq:cflgn}
\inf_{\xi \in \cK^{(n)}} q^{(n)}(\xi)>0\,\text{ and }\,|c^{(n)}(\xi,\eta)|<1\text{ for any $\xi \in \cK^{(n)}$ and $\eta \in \mathcal{N}_{\xi}^{(n)}$.}
\end{equation}
We  endow $\cK^{(n)}$ with the discrete topology and    
let $\mathcal{B}_b(\cK^{(n)})$  denote the space of bounded Borel functions on $\cK^{(n)}$ equipped with the sup-norm.  Under \eqref{eq:cflgn},  $L^{(n)}$ is a bounded linear operator on $\mathcal{B}_b(\cK^{(n)})$ and  generates a continuous-time Markov chain on $\cK^{(n)}$.   The associated semigroup is given by $\{e^{t L^{(n)}}\}_{t> 0}$.    Define $\pi_n \colon C(\R^d) \to \R^{\cK^{(n)}}$ by $
\pi_n f(\xi)=f(\bar{\xi})$.  The next theorem  immediately follows  from Theorems~\ref{thm:1} and \ref{thm:2}.
\begin{theorem}\label{thm:3}
Assume that $D$ is a $C^{1,\alpha}$-domain for some $\alpha \in (0,1]$.  Assume in addition that the following conditions are satisfied:
 \begin{align*}
\lim_{n \to \infty}\sup_{\xi \in \cK^{(n)}}\frac{\eps_n(\xi)}{\rho_n(\xi)}
=0,\quad  \lim_{n \to \infty}\sup_{\xi \in \partial \cK^{(n)}}\frac{\dl_n(\xi)}{\rho_n(\xi)}
=0, \quad\text{and}\quad \lim_{n \to \infty}\sup_{\xi \in \cK^{(n)}} \rho_n(\xi)=0.
\end{align*}
Then,  we have for any $f \in C_{\text{\rm Neu}}^2(\R^d)$,
\begin{equation*}\label{eq:cg}
\lim_{n \to \infty}\sup_{\xi \in \cK^{(n)}}\left| L^{(n)}(\pi_n f)(\xi)- \pi_n \left(\frac{\Delta f}{2} \right)(\xi)   \right|=0.
\end{equation*}
\end{theorem}

In what follows,  we assume the following condition.
\begin{itemize}
\item[{\bf (A)}] For sufficiently large $n \in \N$,  we have $\inf_{\xi \in \cK^{(n)}}\rho_n(\xi)>0.$
\end{itemize}
Under condition~{\bf (A)} and all assumptions of Theorem~\ref{thm:3},  $(L^{(n)},\mathcal{B}_b(\cK^{(n)}))$ is a generator for all sufficiently large $n$.  To achieve a weak convergence of the associated Markov chains,  we may assume that all $(L^{(n)},\mathcal{B}_b(\cK^{(n)}))$ are generators.
For $n \in \N$,  let  $X^{(n)}=(\{X_t^{(n)}\}_{t\ge0 },\{P^{(n)}_\xi\}_{\xi \in \cK^{(n)}})$ be the continuous-time Markov chain on $\cK^{(n)}$ associated with $L^{(n)}$. 
 From Theorem~\ref{thm:3},  we expect that the limit of $\{X^{(n)}\}_{n=1}^\infty$ is the RBM $X=(\{X_t\}_{t \ge 0},\{P_x\}_{x \in \cl{D}})$ on $\cl{D}$.  This is a solution to the following stochastic differential equation:
\begin{equation}\label{eq:skorohod}
X_t=x+B_t+\int_{0}^t\nu(X_s)\,dL_s^X,\quad t \ge 0,\quad \text{$P_{x}$-a.s.\  for any $x \in \cl{D}$},
\end{equation}
where  $B=\{B_t\}_{t \ge 0}$ is a $d$-dimensional Brownian motion  and $L^X=\{L_t^X\}_{t \ge 0}$ is the boundary local time of $X$.  Since $D$ is a $C^{1,\alpha}$-domain,  it is a Lipschitz domain.  Thus,  $X$ can be constructed as a Feller process on $\cl{D}$ (see,  e.g., \cite[Theorem~4.4 and Remark~4.6]{BH} or  \cite[Theorems~2.2 and 2.3]{FT}). 
This means that  the semigroup $\{p_t\}_{t>0}$ of $X$ 
is strongly continuous on $C_\infty(\cl{D})$
and satisfies $p_t (C_\infty(\cl{D})) \subset C_\infty(\cl{D})$ for any $t>0.$ Here,  $C_\infty(\cl{D})$ denotes the space of all continuous functions on $\cl{D}$ vanishing at infinity.  If $D$ is bounded,    $C_\infty(\cl{D})$ coincides with the space $C(\cl{D})$ of all continuous functions on $\cl{D}$.  

Let $d_{\mathrm{H}}$ denote the Hausdorff metric between non-empty compact sets in $\R^d.$
Let $\mathcal{X}$ be  the set of  all non-empty compact subsets of $\cl{D}$ equipped with the metric $d_{\mathrm{H}}$,  and let $\mathcal{D}([0,\infty),\mathcal{X})$ denote the space of right continuous functions on $[0,\infty)$ having left limits and  taking values in $\mathcal{X}$ that is equipped with the Skorohod topology.   For any $n \in \N$,  we have $\cK^{(n)} \subset \mathcal{X}$.  Thus,  each $X^{(n)}$  induces  the law on $\mathcal{D}([0,\infty),\mathcal{X})$.  Any $x \in \cl{D}$ is regarded as an element of $\mathcal{X}$ through the canonical embedding.  Since  $X$ is a diffusion process on $\cl{D}$, the RBM $X$ also induces the law on $\mathcal{D}([0,\infty),\mathcal{X})$.  Let $(\mathcal{L},\text{Dom}(\mathcal{L}))$ denote the generator of the Feller process $X$ on $\cl{D}$.  Define a subspace $C_{c,\text{Neu}}^2(\cl{D})$ of $C_{\infty}(\cl{D})$ as
\begin{align*}
C_{c,\text{Neu}}^2(\cl{D})=\left\{f \in C^2_c(\mathbb{R}^d)|_{\cl{D}} \relmiddle| \langle \nabla f(x),\nu(x) \rangle=0\text{ for every $x \in \partial D$} \right\},
\end{align*}
where $C^2_c(\mathbb{R}^d)=C^2(\R^d) \cap C_c(\R^d)$ and $C_c(\R^d)$ is the space of all compactly supported real-valued continuous functions on $\R^d.$  Let  $f \in C_{c,\text{Neu}}^2(\cl{D})$,  $x \in \cl{D}$,  and $t \ge 0$.  
Applying the It\^{o} formula to \eqref{eq:skorohod},  
we have $P_{x}$-a.s.,
\begin{align*}
&f(X_t)-f(x)\\
&=\int_{0}^{t} \langle \nabla f(X_s),\,dB_s \rangle  +\int_{0}^{t} \langle \nabla f(X_s),\nu(X_s) \rangle\,dL_s^{X} +\frac{1}{2}\int_{0}^{t} \Delta f(X_s)\,ds \\
&=\int_{0}^{t} \langle \nabla f(X_s),\,dB_s \rangle  +0+\frac{1}{2}\int_{0}^{t} \Delta f(X_s)\,ds.
\end{align*}
Taking the expectation on both sides of the above equation with respect to $P_x$, we have  \[
p_tf(x)-f(x)=\frac{1}{2}\int_{0}^t p_s(\Delta f)(x)\,ds.\]
This and the strong continuity of $\{p_t\}_{t>0}$ imply that\[\lim_{t \to 0} \frac{p_tf(x)-f(x)}{t}=\frac{\Delta f}{2}(x)\quad\text{uniformly in $x \in \cl{D}$.}\]    Thus,  $ C_{c,\text{Neu}}^2(\cl{D})$ is a subspace of $\text{Dom}(\mathcal{L})$,  and we have 
$
\mathcal{L}f(x)=(\Delta f/2)(x)$
for any  $x \in \cl{D}$ and $f \in C_{c,\text{Neu}}^2(\cl{D})$.    In particular,  letting $n \in \N$ and $\xi \in \cK^{(n)}$ satisfy $\bar{\xi} \in \cl{D}$,    we have for any $f \in C_{c,\text{Neu}}^2(\cl{D})$,
\begin{equation}\label{eq:LD}
\mathcal{L}f(\bar{\xi})=\frac{\Delta f}{2}(\bar{\xi}).
\end{equation}
The following condition is used to obtain this equation for any sufficiently large $n$ and $\xi \in \cK^{(n)}$.
\begin{itemize}
\item[{\bf (B)}] For sufficiently large $n \in \N$ and any $\xi \in \mathcal{K}^{(n)}$,  we have $\bar{\xi} \in \cl{D}$.
\end{itemize}

Assume condition~{\bf (B)} and take $n \in \N$ such that $\bar{\xi} \in \cl{D}$ holds for any  $\xi \in \mathcal{K}^{(n)}$.  Define  $\widetilde{\pi}_n \colon  C(\cl{D}) \to \R^{\cK^{(n)}}$ by  $(\widetilde{\pi}_nf)(\xi)=f(\bar{\xi})$.  Under condition~{\bf (A)} and all assumptions of Theorem~\ref{thm:3},  we use \eqref{eq:LD} to see that  for any $f \in  C_{c,\text{Neu}}^2(\cl{D})$,
\begin{equation}\label{eq:cg'}
\lim_{n \to \infty}\sup_{\xi \in \cK^{(n)}}\left| L^{(n)}(\widetilde{\pi}_n f)(\xi)- \widetilde{\pi}_n \left( \mathcal{L} f \right)(\xi)   \right|=0.
\end{equation}
This suggests that the laws of $\{X^{(n)}\}_{n=1}^\infty$ converge weakly to the law of $X$ in  $\mathcal{D}([0,\infty),\mathcal{X})$.  The main result of this paper is the proof of this convergence.

\begin{theorem}\label{thm:4}
Assume that all conditions of Theorem~\ref{thm:3} are satisfied.  Furthermore,  assume conditions~{\bf (A)} and  {\bf (B)}  and that $ C_{c,\text{\rm Neu}}^2(\cl{D})$ is a core for $(\mathcal{L},\text{\rm Dom}(\mathcal{L}))$.  Let $x \in \cl{D}$,  and let $\xi^{(n)} \in \cK^{(n)}$ ($n=1,2,\ldots$) satisfy     $\lim_{n \to \infty}d_{\mathrm{H}}(\xi^{(n)},x)=0.$ Then, as $n \to \infty$, the laws of  $\{(X^{(n)},P_{\xi^{(n)}}^{(n)})\}_{n=1}^\infty$ converge weakly in $\mathcal{D}([0,\infty),\mathcal{X})$ to the law of the RBM on $\cl{D}$ starting from $x$.
\end{theorem}

In the following proposition,  we provide sufficient conditions for $C_{c,\text{Neu}}^2(\cl{D})$ to be a core for $(\mathcal{L},\text{\rm Dom}(\mathcal{L}))$.  The proof is given in Appendix~B.

\begin{proposition}\label{prop:core}
If $D$ is a bounded $C^{1,1}$-domain or a bounded convex $C^{1,\alpha}$-domain for some $\alpha \in (0,1]$,  then $ C_{c,\text{\rm Neu}}^2(\cl{D})$ is a core for $(\mathcal{L},\text{\rm Dom}(\mathcal{L})).$
\end{proposition}

\begin{remark}
The space $ C_{c,\text{Neu}}^2(\cl{D})$ is not necessarily a core for $(\mathcal{L},\text{\rm Dom}(\mathcal{L}))$ even if $D$ is a $C^{1,\alpha}$-domain for some $\alpha \in (0,1)$.   Indeed,   let $\varphi \colon \R \to \R$ be a $C^{1,\alpha}$-function such that $\varphi'$ is non-differentiable at any point of $\R$,  and consider $D=\{(x_1,x_2) \in \R^2 \mid x_2>\varphi(x_1)\}$.  Then,  the inward normal unit vector $\nu$ on $\partial D$ is given by \[\nu(x)=\frac{1}{\sqrt{\varphi'(x_1)^2+1}}\left(-\varphi'(x_1),1 \right),\quad x =(x_1,x_2)\in \partial D.\]
Let $f \in  C_{c,\text{Neu}}^2(\cl{D})$.  Since $\langle \nabla f(x), \nu(x) \rangle=0$ for any $x \in \partial D$,  we have
\begin{equation}\label{eq:NU}
\varphi'(x_1)\frac{\partial f}{\partial x_1}(x_1,\varphi(x_1))=\frac{\partial f}{\partial x_2}(x_1,\varphi(x_1)),\quad x_1\in \R.
\end{equation}
Set $O=\{x_1 \in \R \mid   (\partial f/\partial x_1)(x_1,\varphi(x_1))\neq 0\}$ and suppose that $O$ is non-empty.  Since $f$ is a $C^2$-function,  \eqref{eq:NU} shows that $\varphi'$ is a $C^1$-function on $O$,  which contradicts the definition of $\varphi$.  Thus,  $\nabla f =0$ on $\partial D$,  which implies that $f$ is constant on $\partial D$.   Since the support of $f$ is a compact subset of $\R^d$,  we also see that  $f=0$ on $\partial D$.  This implies that $ C_{c,\text{Neu}}^2(\cl{D})$ is not a dense subspace of $C_{\infty}(\cl{D}).$ In particular,  $ C_{c,\text{Neu}}^2(\cl{D})$ is not a core for $(\mathcal{L},\text{\rm Dom}(\mathcal{L}))$.
\end{remark}

In the rest of this section,  we discuss examples that meet the conditions in Theorem~\ref{thm:4}.

\begin{example}\label{ex:voronoi}
Let $D$ be a bounded convex $C^{1,\alpha}$-domain with some $\alpha \in (0,1]$.  Let $\{Z_k\}_{k=1}^\infty$ be $D$-valued independent and identically distributed random variables defined on a probability space $(\Omega,\mathcal{F},P)$.  Assume that the distribution of $Z_1$ is equal to $m|_D/m(D)$.   For $\omega \in \Omega$ and $n \in \N$,  define $
V^{(n)}(\omega)=\{Z_1(\omega),\ldots, Z_n(\omega)\}$.  
Then, for  $P$-a.s.\ $\omega$,  we have $\# V^{(n)}(\omega)=n$ for any $n \in \N$.
For such $\omega$ and an integer $k$ with $1\le k \le n$,  we define a compact subset $\xi^{(k)}(\omega)$ of $ \cl{D}$ as
\[
\xi^{(k)}(\omega)=
\left\{x \in \cl{D}  \relmiddle||x-Z_k(\omega)| \le |x-Z_l(\omega)|\,\, \text{ for any $1\le l \le n$ 
with $l\neq k$} \right\}.
\]
Let $\cK^{(n)}(\omega):=\{\xi^{1}(\omega),\ldots,\xi^{n}(\omega)\}$, which is called the Voronoi partition associated with $V^{(n)}(\omega).$    From \cite[Proposition~5.1]{DS},   there exists $C>0$ such that for $P$-a.s.\  $\omega$ and sufficiently large $n$ (depending on $\omega$),  
\begin{equation}\label{eq:voron}
\max_{\xi \in \cK^{(n)}(\omega)}\text{diam}(\xi) \le C \left(\frac{\log n}{n}\right)^{1/d}.
\end{equation}
Let $\{a_n\}_{n=1}^\infty$ and $\{b_n\}_{n=1}^\infty$ be  two sequences of positive numbers such that 
\begin{equation*}
\lim_{n \to \infty}\left.   \left(\frac{\log n}{n}\right)^{1/d}  \right/  a_n=0,\quad \lim_{n \to \infty}\frac{a_n}{b_n}=0,\quad \text{and }\quad \lim_{n \to \infty}b_n=0.
\end{equation*}
For $\omega  \in \Omega$ and $n \in \N$,  define $\dl_n \colon \cK^{(n)}(\omega) \to (0,\infty)$ and  $\rho_n \colon \cK^{(n)}(\omega) \to (0,\infty)$ by $
\dl_n(\xi)=a_n$ and
\begin{align*}
\rho_n(\xi)&=
\left\{
    \begin{alignedat}{2}
&a_n  &\quad&   \text{ if } \xi \in \cK^{(n)}(\omega) \setminus \partial \cK^{(n)}(\omega) ,  \\
& b_n&  \quad   &    \text{ if }  \xi \in \partial \cK^{(n)}(\omega).
 \end{alignedat}
    \right.
    \end{align*}
Then,  \eqref{eq:voron} implies that for $P$-a.s.\  $\omega$,
 \begin{align*}
&\lim_{n \to \infty}\sup_{\xi \in \cK^{(n)}(\omega) }\frac{\eps_n(\xi)}{\rho_n(\xi)}
=0,\quad  \lim_{n \to \infty}\sup_{\xi \in \partial \cK^{(n)}(\omega) }\frac{\dl_n(\xi)}{\rho_n(\xi)}
=0,\quad\text{and}\\
& \lim_{n \to \infty}\sup_{\xi \in \cK^{(n)}(\omega) } \rho_n(\xi)=0.
\end{align*}
Since  $D$ is a bounded convex domain and each $\cK^{(n)}(\omega)$ is a Voronoi partition,  any $\xi \in \cK^{(n)}(\omega)$ is a convex set.  In particular,  we have  $\bar{\xi} \in \xi \subset  \cl{D}$ for any   $\xi \in \cK^{(n)}(\omega)$.   This implies that condition~{\bf (B)} holds.  Since $D$ is bounded,  condition~{\bf (A)} also holds. From Proposition~\ref{prop:core},  we see that  $ C_{c,\text{Neu}}^2(\cl{D})$ is a core for $(\mathcal{L},\text{\rm Dom}(\mathcal{L}))$.  Therefore,   the assumptions of Theorem~\ref{thm:4} are satisfied.  
\end{example}

\begin{example}\label{ex:srw}
We consider the case of $D=\R^d.$  In this case,  the RBM $X$ is nothing but the $d$-dimensional Brownian motion $B=(\{B_t\}_{t \ge 0},\{P_x\}_{x \in \R^d})$.  For $n \in \N$,  we set \begin{equation}\label{eq:cl}
\mathcal{K}^{(n)}=\left\{  \prod_{i=1}^d \left[x_i-\frac{1}{2n},x_i+\frac{1}{2n} \right] \relmiddle| x=(x_1,\ldots,x_d) \in \frac{1}{n}\mathbb{Z}^d\right\}
\end{equation}
and  let $\dl_n\colon \cK^{(n)} \to (0,\infty)$.  
Because $D=\R^d$,   we have $\partial  \cK^{(n)}=\emptyset$.    Let $\{a_n\}_{n=1}^\infty$ be a sequence of positive numbers  such that 
$\lim_{n \to \infty}n^{-1}/a_n=0$ and $\lim_{n \to \infty}a_n=0$.   For $n \in \N$,  let $\rho_n \colon \cK^{(n)} \to (0,\infty)$  and assume that $\rho_n(\xi)=\dl_n(\xi)=a_n$ for any $\xi \in \cK^{(n)}$.  
Since $\eps_n(\xi)=\sqrt{d}/(2n)$ for any $n \in \N$ and $\xi \in \cK^{(n)}$,  we have
 \[
\lim_{n \to \infty}\sup_{\xi \in \cK^{(n)}}\frac{\eps_n(\xi)}{\rho_n(\xi)} =0\quad\text{and}\quad \lim_{n \to \infty}\sup_{\xi \in \cK^{(n)}} \rho_n(\xi)=0.
\]
Obviously,  conditions~{\bf (A)} and {\bf (B)} hold.  We have
$ C_{c,\text{Neu}}^2(\R^d)= C^2_c(\mathbb{R}^d)$,  and  $C^2_c(\mathbb{R}^d)$ is a core for  the generator of  the Feller process  $B$ on $\R^d.$   Thus,  the assumptions of Theorem~\ref{thm:4} are satisfied.  
\end{example}

\begin{remark}\label{rem:srw}
Let $n \in \N$ and regard $n^{-1}\mathbb{Z}^d$  as a graph with edge set $\{\{x,y\} \mid x,y \in n^{-1}\mathbb{Z}^d,\,  |x-y|\le 1/n \}$.  It is a well-known classical result that the $d$-dimensional Brownian motion can be approximated by  a continuous-time simple random walk on $n^{-1}\mathbb{Z}^d$ with an exponentially distributed waiting time.   In our formulation,   regard $\cK^{(n)}$ given in \eqref{eq:cl} as a graph with edge set $\{\{\xi,\eta\} \mid \xi,\eta \in \cK^{(n)},|\bar{\xi}-\bar{\eta}|\le 1/n\}$.    Then,   $\cl{\cK^{(n)}}(=\{\bar{\xi} \mid \xi \in \cK^{(n)}\})$ coincides with $n^{-1}\mathbb{Z}^d$.
 Let  $\dl_n$ be a positive function on $\cK^{(n)}$.   Since $D=\R^d$,   we have $\partial \cK^{(n)}=\emptyset$.
Let $\rho_n \colon \cK^{(n)} \to (0,\infty)$ satisfy
 $\rho_n(\xi)=\dl_n(\xi)$ for any $\xi \in \cK^{(n)}$.  
As stated in Example~\ref{ex:srw},   we have $\eps_n(\xi)=\sqrt{d}/(2n)$ for any $\xi \in \cK^{(n)}$.    It is easy to verify that the Markov chain $X^{(n)}$ on $\cK^{(n)}$ becomes a continuous-time simple random walk on $\cK^{(n)}$ if and only if $\rho_n\colon \cK^{(n)} \to (0,\infty)$ satisfies $1/n<\rho_n(\xi)<\sqrt{d}/n$ for any   $\xi \in \cK^{(n)}$.  
Unfortunately,  Theorem~\ref{thm:4} does not cover this scaling speed,  in exchange for the generality of the setting.
\end{remark}

\section{Preliminaries}

In this section,  we provide some  basic estimates.  Henceforth,  we fix a partition $\cK$ of $\cl{D}$.  Let $\delta\colon \cK \to (0,1)$ and define $\partial \cK$ as described in \eqref{eq:boundarycell}.  For each $\xi \in \partial \cK$,  we fix $\hat{\xi} \in \partial D$ such that  $|\bar{\xi}-\hat{\xi}|=\text{dist}(\bar{\xi},\partial D).$  
 We also fix   $\rho \colon \cK \to (0,\infty)$ satisfying \eqref{eq:r}. 
The next lemma immediately follows  from the definitions stated above and ones given in \eqref{eq:N1} and \eqref{eq:N2}. 
 
\begin{lemma}\label{lem:ed}
\begin{itemize}
\item[{\rm (1)}]For any $\xi \in \cK$,  $\eta \in \mathcal{N}_\xi$,  and $y \in \eta $,  we have  $|y-\bar{\eta}| \le \eps(\xi)$.  In particular,  $\xi \subset \cl{B}(\bar{\xi},\eps(\xi))$.
\item[{\rm (2)}] For any $\xi \in \partial \cK$,  we have $
|\bar{\xi}-\hat{\xi}| < \dl(\xi).$
\end{itemize}
\end{lemma}

 \begin{lemma}\label{lem:edd}
Let $\xi \in \cK$ and suppose $\rho(\xi)>\eps(\xi)$. Then, we have\[ O_{\xi} \subset \cl{D} \cap B(\check{\xi},\rho(\xi)+\eps(\xi)).\]
It also follows that $
\cl{D} \cap B(\check{\xi},\rho(\xi)-\eps(\xi)) \subset O_{\xi}.$
\end{lemma}
\begin{proof}
Let $\eps_1(\xi)=\max_{\eta \in \mathcal{N}_{\xi} }\sup_{y \in \eta}|y-\bar{\eta}|$.  Let $x \in O_{\xi}$.  Then, there exists $\eta \in \mathcal{N}_\xi$ such that $x \in \eta$. In particular, $x \in \cl{D}$. Lemma~3.1 implies $|x-\bar{\eta}|\le \eps_1(\xi)$. From the definition of $\mathcal{N}_\xi$, we have $|\bar{\eta}-\check{\xi}|<\rho(\xi)$.  Therefore, $x \in B(\check{\xi},\rho(\xi)+\eps_1(\xi))$. Since $\eps(\xi) \ge \eps_1(\xi)$, we see that
 \[ O_{\xi} \subset \cl{D} \cap B(\check{\xi},\rho(\xi)+\eps(\xi)).\]
 Let $\eps_2(\xi)=\max_{\eta \in \widetilde{\mathcal{N}}_{\xi} }\sup_{y \in \eta}|y-\bar{\eta}|.$
Since $\eps(\xi) \ge \eps_2(\xi)$,  we have $\rho(\xi)>\eps_2(\xi)$. Let $x \in \cl{D} \cap B(\check{\xi},\rho(\xi)-\eps_2(\xi))$. Because $\cl{D}=\bigcup_{\eta \in \cK}\eta$, there exists $\eta \in \cK$ such that $x \in \eta$. 
Since $x \in \eta \cap B(\check{\xi},\rho(\xi)-\eps_2(\xi))$,  we have $\eta \in \widetilde{\mathcal{N}}_{\xi}$.  Thus, we have $|x-\bar{\eta}| \le \eps_2(\xi)$ and
\begin{align*}
|\bar{\eta}-\check{\xi}|&\le |\bar{\eta}-x|+|x-\check{\xi}| 
\\
&<\eps_2(\xi)+(\rho(\xi)-\eps_2(\xi)) =\rho(\xi).
\end{align*}
This implies $\eta \in \mathcal{N}_\xi$ and $x \in \eta$. Therefore,  $B(\check{\xi},\rho(\xi)-\eps_2(\xi)) \subset O_\xi$.
Since $\eps(\xi)\ge \eps_2(\xi)$, we  obtain that $
\cl{D} \cap B(\check{\xi},\rho(\xi)-\eps(\xi)) \subset O_{\xi}$.
\end{proof}

 In this section and those that follow,  $C$ (resp.\  $C_D$) denotes a positive constant depending only on $d$ and $c_1$ (resp.\  $D$ and $c_2$). 
For $A,B \subset \R^d$,   let $A \bigtriangleup B$ denote the symmetric difference of $A$ and $B$: $A \bigtriangleup B=(A\setminus B) \cup (B\setminus A)$.  For $x \in \partial D$ and $r>0$,  we define $B_{D}(x,r)=D \cap B(x,r)$.  
\begin{lemma}\label{lem:trngl}
\begin{itemize}
\item[{\rm (1)}] Assume \eqref{eq:A}.  Then,  there exists $C>0$ such that
\[
m\left(O_{\xi} \bigtriangleup B(\bar{\xi},\rho(\xi)) \right)  \le C\eps(\xi)\rho(\xi)^{d-1},\quad \xi \in \cK \setminus  \partial \cK.
\]
\item[{\rm (2)}]  Assume \eqref{eq:B}.  Then,  there exists $C>0$ such that
\[
m\left(O_{\xi} \bigtriangleup B_{D}(\hat{\xi},\rho(\xi)) \right) \le C\eps(\xi)\rho(\xi)^{d-1},\quad \xi \in  \partial \cK.
\]
\end{itemize}
\end{lemma}
\begin{proof}
To simplify the notation,  let $\rho=\rho(\xi)$ and $\eps=\eps(\xi)$.  \\
(1): Since $c_1 \in (0,1)$,  we have $\rho>\eps$. Since 
$\xi \in \cK \setminus \partial \cK$,  we have $B(\bar{\xi},\rho) \subset D$.  Thus,  Lemma~\ref{lem:edd} implies that  \[O_{\xi} \bigtriangleup B(\bar{\xi},\rho) \subset  B(\bar{\xi},\rho+\eps) \setminus B(\bar{\xi},\rho-\eps).\]  Therefore, 
\begin{align}\label{eq:trngl1}
m\left(O_{\xi} \bigtriangleup B(\bar{\xi},\rho) \right)&\le m(B(\bar{\xi},\rho+\eps) \setminus B(\bar{\xi},\rho-\eps))  \\
&=\omega_d  \{(\rho+\eps)^d-(\rho-\eps)^d \}. \notag
\end{align}
We also obtain that
\begin{align}\label{eq:trngl2}
&(\rho+\eps)^d-(\rho-\eps)^d =\int_{\rho-\eps}^{\rho+\eps} (t^{d})' \,dt   \\
&\le \int_{\rho-\eps}^{\rho+\eps} d(\rho+\eps)^{d-1} \,dt=2d\eps(\rho+\eps)^{d-1}\le 2d\eps(1+c_1)^{d-1}\rho^{d-1}. \notag
\end{align}
In the last inequality,  we used \eqref{eq:A}.   The claim follows from \eqref{eq:trngl1} and \eqref{eq:trngl2}.\\
(2):   Since $c_2 \in (0,1)$,  we also obtain  $\rho>\eps$. It follows from Lemma~\ref{lem:edd} that
\begin{align*}O_{\xi} \bigtriangleup B_D(\hat{\xi},\rho) &\subset \cl{D} \cap [B(\hat{\xi},\rho+\eps) \setminus B(\hat{\xi},\rho-\eps)].
\end{align*}
Therefore, 
\[
m(O_{\xi} \bigtriangleup B_D(\hat{\xi},\rho)) \le  m(B(\hat{\xi},\rho+\eps) \setminus B(\hat{\xi},\rho-\eps)).
\]
The rest of the proof is similar to that of (1).
\end{proof}

Let $S^{d-1}:=\{x \in \R^d \mid |x|=1\}$ be the unit sphere in $\R^d$ and $\sigma$ the $(d-1)$-dimensional Hausdorff measure on $\R^d.$ Using polar coordinates,  we obtain that for any $r>0$,
\begin{align}
\fint_{B(0,r)}|y|^2\,m(dy)&=\frac{1}{m(B(0,r))}\int_{S^{d-1}}\left(\int_{0}^{r}s^{2}s^{d-1}\,ds\right)\sigma(du) \label{eq:a0} \\
&=\frac{1}{\omega_d r^d} \frac{\sigma(S^{d-1})r^{d+2}}{d+2}=\frac{d r^2}{d+2}. \notag  
\end{align}
In the last equation,  we used the fact that $\sigma(S^{d-1})/\omega_d=d$.   We will use  \eqref{eq:a0} to prove the following lemma.
\begin{lemma}\label{lemma:span}
Assume \eqref{eq:A}.  Then,  there exists $C>0$ such that for any $\xi \in \cK \setminus  \partial \cK$,
\begin{align}
& \min_{u \in S^{d-1}}\langle u,  Q(\xi) u \rangle \ge C\rho(\xi)^2, \quad q(\xi) \ge C \rho(\xi)^2, \quad \text{and} \quad |b(\xi)|\le C\eps(\xi).  \label{eq:span}
 \end{align}
\end{lemma}

\begin{proof}
Let $\xi \in \cK \setminus  \partial \cK$ and let $u \in S^{d-1}$.  A direct computation shows that
\begin{align}
 &\langle u, Q(\xi)u  \rangle   =\sum_{j=1}^{N_\xi} \langle  \overline{\eta^{(j)}}-\bar{\xi},u \rangle^2 \frac{m(\eta^{(j)})}{m(O_{\xi})} =\sum_{\eta \in \mathcal{N}_{\xi}}\langle \bar{\eta}-\bar{\xi},u \rangle^2 \frac{m(\eta)}{m(O_{\xi})}.
 \label{eq:a1}
\end{align}
For any $\eta \in \mathcal{N}_\xi$, we have  
\begin{align}\label{eq:a2}
&\fint_{\eta}(\langle y-\bar{\xi},u \rangle^2- \langle y-\bar{\eta},u \rangle^2)\,m(dy)  \\
&=\fint_{\eta}(\langle y,u \rangle^2-2\langle y,u \rangle\langle \bar{\xi},u \rangle+\langle \bar{\xi},u \rangle^2)\,m(dy) \notag \\
&\quad- \fint_{\eta}(\langle y,u \rangle^2-2\langle y,u \rangle\langle \bar{\eta},u \rangle+\langle \bar{\eta},u \rangle^2)\,m(dy)  \notag \\
&=-2\langle \bar{\eta},u \rangle\langle \bar{\xi},u \rangle+\langle \bar{\xi},u \rangle^2+2\langle \bar{\eta},u \rangle^2+\langle \bar{\eta},u \rangle^2 \notag \\
&=\langle \bar{\eta}-\bar{\xi},u \rangle^2. \notag
\end{align}
In the second equation,  we used the fact that $\bar{\eta}$ is the center of gravity of $\eta$.  From \eqref{eq:a1} and  \eqref{eq:a2},  it follows that
\begin{align*}
\langle u, Q(\xi)u  \rangle&=\sum_{\eta \in \mathcal{N}_{\xi}} \left(\fint_{\eta}(\langle y-\bar{\xi},u \rangle^2- \langle y-\bar{\eta},u \rangle^2)\,m(dy) \right)\frac{m(\eta)}{m(O_{\xi})} \\
&=\fint_{O_\xi}\langle y-\bar{\xi},u \rangle^2\,m(dy) -\sum_{\eta \in \mathcal{N}_{\xi}}\left(\fint_{\eta}  \langle y-\bar{\eta},u \rangle^2\,m(dy) \right) \frac{m(\eta)}{m(O_{\xi})}.
\end{align*}
Let $\rho=\rho(\xi)$ and $\eps=\eps(\xi)$. 
From Lemma~\ref{lem:ed}~(1) and  the above equation,
\begin{align} 
&\langle u,Q(\xi)u \rangle \ge \fint_{O_{\xi}}\langle y-\bar{\xi},u \rangle^2\,m(dy)-\eps^2. \label{eq:a3}
\end{align}
From Lemma~\ref{lem:edd},   we have  $
O_{\xi} \supset B(\bar{\xi},\rho-\eps)$
and 
\begin{align}\label{eq:a4}
\int_{O_{\xi}}\langle y-\bar{\xi},u \rangle^2\,m(dy) &\ge  \int_{B(\bar{\xi},\rho-\eps) }\langle y-\bar{\xi},u \rangle^2\,m(dy) \\
&=\frac{1}{d}\int_{B(\bar{\xi},\rho-\eps)}| y-\bar{\xi}|^2\,m(dy) .\notag
\end{align}
In the second line,  we used the fact that $u \in S^{d-1}.$
Lemma~\ref{lem:edd} also implies
\begin{align}\label{eq:a5}
O_\xi \subset B(\bar{\xi},\rho+\eps).
\end{align}
From \eqref{eq:a0},  \eqref{eq:a3},  \eqref{eq:a4},  and \eqref{eq:a5},  it follows that
\[
 \langle u, Q(\xi)u  \rangle  \ge  \frac{1}{d+2}  \frac{(\rho-\eps)^{d+2}}{(\rho+\eps)^d}-  \eps^2=a_1(\eps/\rho)\rho^2 .\]
We now use \eqref{eq:A} to see that there exists $C>0$ such that $
\langle u,Q(\xi)u \rangle \ge  C\rho^2.$
Since $u \in S^{d-1}$ is arbitrarily chosen,  this implies  the first inequality of \eqref{eq:span}.  This also implies the second inequality  since $q(\xi)$ is equal to the sum of the eigenvalues of $Q(\xi)/d$.

We lastly prove the third inequality of \eqref{eq:span}.  Note that \[
b(\xi)=\sum_{\eta \in \mathcal{N}_{\xi}} (\bar{\eta}-\bar{\xi}) \frac{m(\eta)}{m(O_{\xi})}=\fint_{O_\xi}(y-\bar{\xi})\,m(dy).\]
Therefore,  we obtain that 
\begin{align*}
\left|b(\xi)\right| &=\left|\fint_{O_{\xi}}(y-\bar{\xi})\,m(dy)-0\right|=\left|\fint_{O_{\xi}}(y-\bar{\xi})\,m(dy)-\fint_{B(\bar{\xi},\rho)}(y-\bar{\xi})\,m(dy)\right|\\
&\le  \left|\frac{1}{m(O_{\xi})}-\frac{1}{m(B(\bar{\xi},\rho))} \right|\int_{O_{\xi}}|y-\bar{\xi}|\,m(dy)\\
&\quad + \frac{1}{m\left(B(\bar{\xi},\rho)\right)} \left|\int_{O_{\xi} }(y-\bar{\xi})\,m(dy) -\int_{ B(\bar{\xi},\rho)}(y-\bar{\xi})\,m(dy) \right|.
\end{align*}
This implies that
\begin{equation}\label{eq:a8''}
\left|b(\xi)\right| \le  \frac{2m(O_{\xi} \bigtriangleup B(\bar{\xi},\rho))}{m(B(\bar{\xi},\rho))} \sup_{y \in O_\xi \cup B(\bar{\xi},\rho)}|y-\bar{\xi}|.
\end{equation}
From  \eqref{eq:A} and \eqref{eq:a5},  we have $O_\xi \cup B(\bar{\xi},\rho) \subset B(\bar{\xi},(1+c_1)\rho)$  and  \[\sup_{y \in O_\xi \cup B(\bar{\xi},\rho)}|y-\bar{\xi}| \le (1+c_1)\rho.\]  Applying this and Lemma~\ref{lem:trngl}~(1) to \eqref{eq:a8''},  we  complete the proof.
\end{proof}

\begin{lemma}\label{lem:span2}
For any $\xi \in \partial \cK$ and $u \in S^{d-1}$,
\begin{align*}
\langle u, Q(\xi)u \rangle &\ge \fint_{O_{\xi}}  \langle y-\hat{\xi},u \rangle^2 \,m(dy) -\eps(\xi)^2 -2\eps(\xi)\rho(\xi).
 \end{align*}
\end{lemma}
\begin{proof}
Let $\xi \in  \partial \cK$ and $u \in S^{d-1}$.  
A direct computation shows that 
\begin{align*}
&\langle u, Q(\xi) u  \rangle \\
&=\sum_{\eta \in \mathcal{N}_{\xi}}\langle \bar{\eta}-\hat{\xi},u \rangle^2 \frac{m(\eta)}{m(O_{\xi})}  =\sum_{\eta \in \mathcal{N}_{\xi}}\langle \bar{\eta}-\bar{\xi}+\bar{\xi}-\hat{\xi},u \rangle^2 \frac{m(\eta)}{m(O_{\xi})} \\
&=\sum_{\eta \in \mathcal{N}_{\xi}}\langle \bar{\eta}-\bar{\xi},u \rangle^2 \frac{m(\eta)}{m(O_{\xi})}+2\langle \bar{\xi}-\hat{\xi},u \rangle\sum_{\eta \in \mathcal{N}_{\xi}}\langle \bar{\eta}-\bar{\xi},u \rangle  \frac{m(\eta)}{m(O_{\xi})}+ \langle \bar{\xi}-\hat{\xi},u \rangle^2.
\end{align*}
We note that \eqref{eq:a2} holds even when $\xi \in \partial \cK$.  Thus,  we have for any $\eta \in \mathcal{N}_\xi$,
\begin{align*}
&\langle \bar{\eta}-\bar{\xi},u \rangle^2 =\fint_{\eta}(\langle y-\bar{\xi},u \rangle^2- \langle y-\bar{\eta},u \rangle^2)\,m(dy).
\end{align*} 
It also follows that
\begin{align*}
&\langle \bar{\eta}-\bar{\xi},u \rangle=\fint_{\eta}(\langle y-\bar{\xi},u \rangle- \langle y-\bar{\eta},u \rangle)\,m(dy) \notag. 
\end{align*}
Combining the above four equations gives
\begin{align*}
&\langle u, Q(\xi) u  \rangle\\
&=  \sum_{\eta \in \mathcal{N}_{\xi}}\left( \fint_{\eta}\langle y-\bar{\xi},u \rangle^2\,m(dy) \right) \frac{m(\eta)}{m(O_{\xi})}-\sum_{\eta \in \mathcal{N}_{\xi}} \left( \fint_{\eta} \langle y-\bar{\eta},u \rangle^2\,m(dy) \right) \frac{m(\eta)}{m(O_{\xi})} \\
&\quad +2\langle \bar{\xi}-\hat{\xi},u \rangle\sum_{\eta \in \mathcal{N}_{\xi}} \left(\fint_{\eta}\langle y-\bar{\xi},u \rangle\,m(dy)  \right) \frac{m(\eta)}{m(O_{\xi})} \\
&\quad -2\langle \bar{\xi}-\hat{\xi},u \rangle\sum_{\eta \in \mathcal{N}_{\xi}} \left(\fint_{\eta} \langle y-\bar{\eta},u \rangle\,m(dy)  \right) \frac{m(\eta)}{m(O_{\xi})} +\langle \bar{\xi}-\hat{\xi},u \rangle^2.
\end{align*}
The sum of the first, third, and fifth terms of the right-hand side is equal to
\begin{align*}
&\fint_{O_{\xi}} \langle y-\bar{\xi},u \rangle^2\,m(dy)+2\langle \bar{\xi}-\hat{\xi},u \rangle \fint_{O_{\xi}} \langle y-\bar{\xi},u \rangle\,m(dy) +\langle \bar{\xi}-\hat{\xi},u \rangle^2 \\
&=\fint_{O_{\xi}} \left( \langle y-\bar{\xi},u \rangle+\langle \bar{\xi}-\hat{\xi},u \rangle \right)^2\,m(dy)=\fint_{O_{\xi}}  \langle y-\hat{\xi},u \rangle^2 \,m(dy).
\end{align*}
Thus,  it suffices to estimate the remainder terms.  From Lemma~\ref{lem:ed},
\begin{align*}
& \left| \sum_{\eta \in \mathcal{N}_{\xi}} \left[\fint_{\eta} \langle y-\bar{\eta},u \rangle^2\,m(dy) +2\langle \bar{\xi}-\hat{\xi},u \rangle \left(\fint_{\eta} \langle y-\bar{\eta},u \rangle\,m(dy)  \right) \right] \frac{m(\eta)}{m(O_{\xi})} \right| \notag \\
& \le \eps(\xi)^2+2\dl(\xi)\eps(\xi).
\end{align*}
Since $\dl(\xi)\le \rho(\xi)$,  we obtain the claim. 
\end{proof}

If $D$ is a $C^{1,\alpha}$-domain for some $\alpha  \in (0,1]$,  it satisfies the uniform cone condition.  This implies that there exists $C_D>0$ such that  
\begin{equation}\label{eq:D}
m\left(D \cap B(x,r) \right) \ge C_Dm\left(B(x,r)\right),\quad x \in \partial D,\,\,r \in (0,1).
\end{equation}
This estimate will be used to prove the following lemma.

\begin{lemma}\label{lem:span22}
Assume that $D$ is a $C^{1,\alpha}$-domain for some $\alpha \in (0,1]$.  Assume in addition that \eqref{eq:B} holds.  Then,  there exists $C_D>0$ such that for any $\xi\in   \partial \cK$,  \begin{align*}
 &\min_{u \in S^{d-1}}\langle u, Q(\xi) u \rangle \ge C_D \rho(\xi)^2, \quad q(\xi) \ge C_D \rho(\xi)^2, \\
&\text{and} \quad \frac{|b(\xi)|}{\rho(\xi)}\le C_D \left( \frac{\eps(\xi)}{\rho(\xi)}+\rho(\xi)^\alpha \right).
 \end{align*}
\end{lemma}

\begin{proof}
Let $\xi \in \partial \cK$  and $\rho=\rho(\xi)$.  From the definition of $R_D$, there exists $e \in O_\xi $ such that   $B(e,\rho/R_D) \subset O_{\xi}$.   Let $r=\rho/R_D$ and $B=B(e,r)$. For $u \in S^{d-1}$,  we have 
\begin{align}\label{eq:a10} 
&\int_{O_{\xi}}  \langle y-\hat{\xi},u \rangle^2 \,m(dy)\ge \int_{B} \langle y-\hat{\xi},u \rangle^2 \,m(dy) \\
&=\int_{B} \langle (y-e)-(\hat{\xi}-e),u \rangle^2 \,m(dy)   \notag \\
&=\int_{B}  \langle y-e,u \rangle^2\,m(dy) -  2\langle \hat{\xi}-e,u \rangle \int_{B}  \langle y-e,u \rangle \,m(dy) \notag \\
&\quad+\int_{B}  \langle \hat{\xi}-e,u \rangle^2  \,m(dy)  \notag \\
&= \int_{B}  \langle y-e,u \rangle^2\,m(dy) +0 +\int_{B}  \langle \hat{\xi}-e,u \rangle^2  \,m(dy) \notag \\
&\ge  \int_{B}  \langle y-e,u \rangle^2\,m(dy)=\frac{1}{d}\int_{B}|y-e|^2\,m(dy). \notag 
\end{align}
Let $\eps=\eps(\xi)$.  Lemma~\ref{lem:edd} implies $O_{\xi} \subset B(\hat{\xi},\rho+\eps)$.  Thus,  using \eqref{eq:a0} and \eqref{eq:a10},  we have
\begin{align}
\fint_{O_{\xi}}  \langle y-\hat{\xi},u \rangle^2 \,m(dy)
&\ge  \frac{m(B(e,r))}{m( \cl{B}(\hat{\xi},\rho+\eps))} \frac{r^2}{d+2}= \frac{1}{d+2}\frac{ (\rho/R_D)^{d+2}}{(\rho+\eps)^d} .\label{eq:a12}
\end{align}
From Lemma~\ref{lem:span2} and \eqref{eq:a12},  it follows that
\begin{align*}
\langle u, Q(\xi) u \rangle   
\notag \ge  \frac{1}{d+2}\frac{ (\rho/R_D)^{d+2}}{(\rho+\eps)^d} -\eps^2 -2\eps \rho=a_2(\eps/\rho)\rho^2.
 \end{align*}
Since $u \in S^{d-1}$ is arbitrarily chosen,  this and \eqref{eq:B} imply that there exists $C_D>0$ such that \[\min_{u \in S^{d-1}}\langle u, Q(\xi) u \rangle \ge C_D\rho^2.\]
 Since  $q(\xi)$ is equal to the sum of the eigenvalues of $Q(\xi)/d$,   we also see that  $q(\xi)\ge C_D \rho^2$.   From the same argument as in the proof of Lemma~\ref{lemma:span},  it follows that
\begin{align*}
\left|b(\xi) \right|&=\left|\fint_{O_{\xi}}(y-\hat{\xi})\,m(dy)- \fint_{B_{D}(\hat{\xi},\rho)}(y-\hat{\xi})\,m(dy)\right| \notag  \\
&\quad +\left| \fint_{B_{D}(\hat{\xi},\rho)}(y-\hat{\xi})\,m(dy)-\beta_d\rho\nu(\hat{\xi}) \right| \notag \\
& \le 2\left(1+c_2\right)\rho \frac{m(O_{\xi} \bigtriangleup B_D(\hat{\xi},\rho))}{m(B_D(\hat{\xi},\rho))} +\left| \fint_{B_{D}(\hat{\xi},\rho)}(y-\hat{\xi})\,m(dy)-\beta_d \rho\nu(\hat{\xi}) \right|.
\end{align*}
Applying Lemmas~\ref{lem:trngl}~(2) and \ref{lem:ub}  to the above inequality,  we have $C_D>0$ such that
\begin{align*}
\frac{\left|b(\xi) \right|}{\rho}& \le  \frac{C_D \eps \rho^{d-1}}{m(B_D(\hat{\xi},\rho))}+C_D \rho^{\alpha}.
\end{align*}
 This and \eqref{eq:D} imply the third desired inequality.
\end{proof}

Let $\xi \in \cK$,  and  let $ \mathcal{N}_{\xi}  $ be expressed  as  $\mathcal{N}_{\xi}=\{\eta^{(1)},\ldots,\eta^{(N_\xi)}\}$.   Let  $\widetilde{A}(\xi)$ be the  $d \times N_\xi$ matrix  whose $(i,j)$-th element is given by
\begin{align*}
        &(\cl{\eta^{(j)}}_i-\check{\xi}_i) \sqrt{\frac{m(\eta^{(j)})}{m(O_{\xi})}}.
\end{align*}
We write $\widetilde{A}^{+}(\xi)$ for the generalized inverse matrix of $\widetilde{A}(\xi)$.   Suppose \eqref{eq:A} or \eqref{eq:B} accordingly as $\xi \in \cK \setminus \partial \cK$ or $\xi \in \partial \cK$.  Since  $\widetilde{A}(\xi){}^{\,t\!}\widetilde{A}(\xi)=Q(\xi)$,   Lemmas~\ref{lemma:span} and \ref{lem:span22}  imply  that the rank of $\widetilde{A}(\xi){}^{\,t\!}\widetilde{A}(\xi)$ is equal to $d$.  Therefore,  we have
\[\widetilde{A}^{+}(\xi)={}^{\,t\!}\widetilde{A}(\xi) (\widetilde{A}(\xi) {}^{\,t\!}\widetilde{A}(\xi))^{-1}={}^{\,t\!}\widetilde{A}(\xi)Q(\xi)^{-1}.
\]
Let $D(\xi)$ denote the diagonal matrix of size $N_\xi$ whose $(k,k)$-th element is given by 
$\sqrt{m(\eta^{(k)})/m(O_{\xi})}$.  Then,  we have $A(\xi)=\widetilde{A}(\xi)D(\xi)$ and 
\[
A^{+}(\xi)=D(\xi)^{-1} \widetilde{A}^{+}(\xi)=(D(\xi)^{-1})({}^{\,t\!}\widetilde{A}(\xi)) Q(\xi)^{-1}.
\]
 Applying this equation to \eqref{eq:c0},  we obtain that
\begin{align}
c(\xi)&=A^{+}(\xi)b(\xi)=(D(\xi)^{-1})({}^{\,t\!}\widetilde{A}(\xi)) Q(\xi)^{-1} b(\xi).    \label{eq:a+}
\end{align}
If $\xi \in \mathcal{K} \setminus \partial \cK$, the $(k,i)$-th element of $(D(\xi)^{-1}) ({}^{\,t\!}\widetilde{A}(\xi))$ is equal to 
\begin{equation}\label{eq:a++}
\sqrt{\frac{m(O_\xi)}{m(\eta^{(k)})}} \times (\overline{\eta^{(k)}}_i-\bar{\xi}_i) \sqrt{\frac{m(\eta^{(k)})}{m(O_{\xi})}}=\overline{\eta^{(k)}}_i-\bar{\xi}_i.
\end{equation}
If $\xi \in  \partial \cK$, the $(k,i)$-th element  is equal to 
\begin{equation}\label{eq:a+++}
\sqrt{\frac{m(O_\xi)}{m(\eta^{(k)})}} \times (\overline{\eta^{(k)}}_i-\hat{\xi}_i) \sqrt{\frac{m(\eta^{(k)})}{m(O_{\xi})}}=\overline{\eta^{(k)}}_i-\hat{\xi}_i. 
\end{equation}

We now prove Lemmas~\ref{lem:A} and \ref{lem:B}.
\begin{proof}[Proof of Lemma~\ref{lem:A}]
From Lemma~\ref{lemma:span},  it suffices to prove the first inequality.  Let $\xi \in \cK \setminus  \partial \cK.$ 
From the first and third inequalities of Lemma~\ref{lemma:span},  it follows that the norm of $ Q(\xi)^{-1} b(\xi)$ is bounded above by $C\eps(\xi)/\rho(\xi)^2$ for some $C>0$.   From \eqref{eq:a++},  the absolute value of any entry of $(D(\xi)^{-1}) ({}^{\,t\!}\widetilde{A}(\xi))$ is bounded above by $\rho(\xi)$.  These facts and \eqref{eq:a+} imply $\max_{\eta \in \mathcal{N}_{\xi}}|c(\xi,\eta)|  \le  C \eps(\xi)/\rho(\xi) $. 
\end{proof}

The proof of Lemma~\ref{lem:B} is similar to that of Lemma~\ref{lem:A}.
\begin{proof}[Proof of Lemma~\ref{lem:B}]
From Lemma~\ref{lemma:span},  it suffices to prove the first inequality,  which is obtained by combining \eqref{eq:a+},   \eqref{eq:a+++},  and the first and third inequalities of Lemma~\ref{lemma:span}. 
\end{proof}

Define $\widetilde{Q} \colon \cK \to \R^d \otimes \R^d$ and $\widetilde{q}\colon \cK \to \R$ by \[
\widetilde{Q}(\xi)=\fint_{O_{\xi}}(y-\check{\xi})\otimes (y-\check{\xi})\,m(dy),\quad \widetilde{q}(\xi)=\frac{1}{d}\text{Tr}[\widetilde{Q}(\xi)].\]

\begin{lemma}\label{lem:qtq}
For any $\xi \in \cK$, $
|q(\xi)-\widetilde{q}(\xi)| \le (\eps(\xi)+2\rho(\xi))\eps(\xi).$
\end{lemma}
\begin{proof}
Let $\xi \in \cK$.    It is straightforward to see that
\begin{align}\label{eq:qtq}
&\text{Tr}[Q(\xi)]-\text{Tr}[\widetilde{Q}(\xi)] \\
&= \sum_{i=1}^d \left(\sum_{\eta \in \mathcal{N}_{\xi}} (\bar{\eta}_i-\check{\xi}_i)^2\, \frac{m(\eta)}{m(O_{\xi})}  -\fint_{O_{\xi}}(y_i-\check{\xi}_i)^2\,m(dy)\right) \notag \\
&= \sum_{i=1}^d  \sum_{\eta \in \mathcal{N}_{\xi}}  \left[(\bar{\eta}_i-\check{\xi}_i)^2 -\fint_{\eta}(y_i-\check{\xi}_i)^2\,m(dy) \right]\frac{m(\eta)}{m(O_{\xi})}  .  \notag
\end{align}
Let $\eta \in \mathcal{N}_\xi$,  and let $i \in \N$ satisfy  $1\le i\le d$. Then, we have
\begin{align}\label{eq:qtq2}
&(\bar{\eta}_i-\check{\xi}_i)^2 -\fint_{\eta}(y_i-\check{\xi}_i)^2\,m(dy) \\
&=(\bar{\eta}_i-\check{\xi}_i)^2- \fint_{\eta}(y_i-\check{\xi}_i)(\bar{\eta}_i-\check{\xi}_i)\,m(dy) \notag \\
&\quad+\fint_{\eta}(y_i-\check{\xi}_i)(\bar{\eta}_i-\check{\xi}_i)\,m(dy) -\fint_{\eta}(y_i-\check{\xi}_i)^2\,m(dy)  \notag \\
&=0+\fint_{\eta}(y_i-\check{\xi}_i)(\bar{\eta}_i-y_i)\,m(dy).   \notag
\end{align}  
In the last equation,  we used the fact that $\bar{\eta}$ is the center of gravity of $\eta.$   We also use  Lemma~\ref{lem:ed}~(1) to see that
\begin{align}
|(y_i-\check{\xi}_i)(\bar{\eta}_i-y_i)| &\le (|y_i-\bar{\eta}_i|+|\bar{\eta}_i-\bar{\xi}_i|+|\bar{\xi}_i-\check{\xi}_i|)|\bar{\eta}_i-y_i| \label{eq:add3} \\
&\le  (\eps(\xi)+\rho(\xi)+\dl(\xi))\eps(\xi)  \notag \\
&\le (\eps(\xi)+2\rho(\xi))\eps(\xi). \notag
\end{align}
Using  \eqref{eq:qtq2}, \eqref{eq:qtq}, and \eqref{eq:add3}, we obtain
\begin{align*}
|\text{Tr}[Q(\xi)]-\text{Tr}[\widetilde{Q}(\xi)]| &\le \sum_{i=1}^d  \left(\sum_{\eta \in \mathcal{N}_{\xi}} \fint_{\eta}|y_i-\check{\xi}_i||\bar{\eta}_i-y_i|\,m(dy) \right)\frac{m(\eta)}{m(O_{\xi})} \\
&\le d(\eps(\xi)+2\rho(\xi))\eps(\xi).
\end{align*}
Dividing  both sides of the above inequality by $d$,  we obtain the conclusion.
\end{proof}

\begin{lemma}\label{lem:qr}
Assume that \eqref{eq:A} holds.  Then,  there exists $C>0$ such that for any $\xi \in \cK \setminus \partial \cK$,
\[
\left|q(\xi)-\frac{ \rho(\xi)^{2} }{d+2}  \right| \le C\eps(\xi)\rho(\xi).
\]
\end{lemma}
\begin{proof}
Let $\xi \in \cK \setminus  \partial \cK$ and $\rho=\rho(\xi)$.  Using \eqref{eq:a0},  we obtain 
\begin{align} \label{eq:qr}
&d\left| \widetilde{q}(\xi)-\frac{ \rho^{2}}{d+2} \right| =\left| \fint_{O_\xi}|y-\bar{\xi}|^2\,m(dy)-\fint_{B(\bar{\xi},\rho)}|y-\bar{\xi}|^2\,m(dy)  \right|  \\
&\le \left|\frac{1}{m(O_{\xi})}-\frac{1}{m(B(\bar{\xi},\rho))}\right|  \int_{O_\xi} |y-\bar{\xi}|^2 \,m(dy) \notag  \\
&\quad + \frac{1}{m(B(\bar{\xi},\rho))} \left|  \int_{O_{\xi}  } |y-\bar{\xi}|^2 \,m(dy)-\int_{B(\bar{\xi},\rho)} |y-\bar{\xi}|^2  \,m(dy) \right|  \notag \\
&\le \frac{2m(O_{\xi} \bigtriangleup B(\bar{\xi},\rho))}{m(B(\bar{\xi},\rho))}  \sup_{y \in O_{\xi}  \cup B(\bar{\xi},\rho)} |y-\bar{\xi}|^2 . \notag
\end{align}
By using  \eqref{eq:A} and Lemma~\ref{lem:edd},  we  have  $O_{\xi}  \cup B(\bar{\xi},\rho) \subset B(\bar{\xi},(1+c_1)\rho)$ and 
\begin{equation}\label{eq:qrb}
\sup_{y \in O_{\xi}  \cup B(\bar{\xi},\rho)} |y-\bar{\xi}|^2\le (1+c_1)^2\rho^2.
\end{equation}
From \eqref{eq:qr},  \eqref{eq:qrb},   and Lemma~\ref{lem:trngl}~(1),  
we have \[
\left|\widetilde{q}(\xi)-\frac{  \rho(\xi)^{2} }{d+2}  \right| \le C\eps(\xi)\rho(\xi)
\]
for some $C>0$.  This and Lemma~\ref{lem:qtq} complete the proof.
\end{proof}

\begin{lemma}\label{lem:qr2}
Assume that $D$ is a $C^{1,\alpha}$-domain for some $\alpha \in (0,1]$.  Assume in addition that \eqref{eq:B} holds.  Then,  there exists $C_D>0$ such that for any $\xi \in  \partial \cK$ with $\rho(\xi)<R$,
\[
\left|q(\xi)-\frac{\rho(\xi)^{2} }{d+2}\right| \le C_D\left(\eps(\xi)\rho(\xi)+\rho(\xi)^{2+\alpha}\right),\]
where  $R$ is the positive constant in Definition~\ref{defn:Dc1a}.
\end{lemma}
\begin{proof}
Let $\xi \in \partial \cK$ with $\rho(\xi)<R$.  Let $\rho=\rho(\xi)$ and define \[
B_{+}(\hat{\xi},\rho)=\{y=(y',y_d) \in B(0,\rho) \text{ in $CS_{\hat{\xi}}$} \mid y_d>0\},\] where   $CS_{\hat{\xi}}$ is an orthonormal coordinate system with its origin at $\hat{\xi}$. By using Lemma~\ref{lem:uph2},  we obtain that
\[
d\left| \widetilde{q}(\xi)-\frac{\rho^{2}  }{d+2}\right|=\left| \fint_{O_\xi}|y-\bar{\xi}|^2\,m(dy)-\fint_{B_{+}(\hat{\xi},\rho)}|y-\hat{\xi}|^2\,m(dy)  \right|.
\]
An argument similar to the one in Lemma~\ref{lem:qr}  shows that there exists $C_D>0$ such that
\begin{equation}\label{eq:qr2}
\left| \widetilde{q}(\xi)-\frac{\rho^{2}}{d+2} \right| \le \frac{C_D m(O_{\xi} \bigtriangleup B_{+}(\hat{\xi},\rho))}{m(B_{+}(\hat{\xi},\rho))} \rho^2.
\end{equation}
Lemmas~\ref{lem:trngl} and \ref{lem:ub}  imply that there exists $C_D>0$ such that
\begin{align}\label{eq:qr3}
m(O_{\xi} \bigtriangleup B_{+}(\hat{\xi},\rho)) &\le m(O_{\xi} \bigtriangleup B_{D}(\hat{\xi},\rho))+m(B_{D}(\hat{\xi},\rho) \bigtriangleup B_{+}(\hat{\xi},\rho))   \\
&\le C_D (\eps \rho^{d-1}+\rho^{d+\alpha})  . \notag
\end{align} 
Note that $m(B_{+}(\hat{\xi},\rho))=\omega_d \rho^d/2$.
Thus,  applying  \eqref{eq:qr3}  to  \eqref{eq:qr2} and using Lemma~\ref{lem:qtq},  we obtain the conclusion.
\end{proof}

\begin{lemma}\label{lem:Q/q}
Assume \eqref{eq:A}.   Then,  there exists $C>0$ such that
\begin{align*}
 \left \|\frac{\widetilde{Q}(\xi)}{q(\xi)} -I_d \right\| \le  \frac{C\eps(\xi)}{\rho(\xi)},\quad \xi \in \cK \setminus \partial \cK.
\end{align*}
\end{lemma}
\begin{proof}
Let $\xi \in \cK \setminus  \partial \cK$ and let $\rho=\rho(\xi)$.  We have
\begin{equation}\label{eq:Q/q}
\widetilde{Q}(\xi) -q(\xi) I_d   =\left(\widetilde{Q}(\xi) -\frac{ \rho^2 }{d+2}I_d \right)-\left(q(\xi)I_d-\frac{\rho^{2}}{d+2} I_d  \right).
\end{equation}
Let $\eps=\eps(\xi)$.  Applying the third inequality of Lemma~\ref{lemma:span} and Lemma~\ref{lem:qr} to the above equation,  we see that there exists $C>0$ such that
\begin{equation}\label{eq:Q/q1}
 \left \|\frac{\widetilde{Q}(\xi)}{q(\xi)} -I_d\right\| \le \frac{C}{\rho^2}\left\| \widetilde{Q}(\xi) -\frac{\rho^2 }{d+2}I_d  \right \|+\frac{C\eps}{\rho}.
\end{equation}
We also see from \eqref{eq:a0} that\[\fint_{B(\bar{\xi},\rho)} (y-\bar{\xi}) \otimes (y-\bar{\xi}) \,m(dy) =\frac{\rho^{2}}{d+2} I_d. \]
From this identity and an argument similar to that in Lemma~\ref{lem:qr},   it follows that
\begin{align*}
\left \|\widetilde{Q}(\xi)- \frac{\rho^{2}}{d+2} I_d \right\| &\le \frac{2m(O_{\xi} \bigtriangleup B(\bar{\xi},\rho))}{m(B(\bar{\xi},\rho))}  \sup_{y \in O_{\xi}  \cup B(\bar{\xi},\rho)} \| (y-\bar{\xi})\otimes (y-\bar{\xi})\| \\
&\le   C\eps\rho
\end{align*}
for some $C>0$.  Combining this inequality and \eqref{eq:Q/q1},   we obtain the conclusion.
\end{proof}

\begin{lemma}\label{lem:Q/q2}
Assume  that $D$ is a $C^{1,\alpha}$-domain for some $\alpha  \in (0,1]$.   Assume in addition that  \eqref{eq:B} holds.   Then,  there exists $C_D>0$ such that for any $\xi \in \partial \cK$ with $\rho(\xi)<R$,
\begin{align*}
 \left \|\frac{\widetilde{Q}(\xi)}{q(\xi)} -I_d \right\| \le  C_D  \left( \frac{\eps(\xi)}{\rho(\xi)}+\rho(\xi)^\alpha \right),
\end{align*}
where  $R$ is the positive constant in Definition~\ref{defn:Dc1a}.
\end{lemma}
\begin{proof}
Let $\xi \in  \partial \cK$ with $\rho(\xi)<R$.   
Note that \eqref{eq:Q/q} holds even for $\xi$ in $\partial \cK$.
Thus,  using the third inequality of Lemma~\ref{lemma:span} and Lemma~\ref{lem:qr2},  we see that there exists $C_D>0$ such that
\[
 \left \|\frac{\widetilde{Q}(\xi)}{q(\xi)} -I_d \right\| \le \frac{C_D}{\rho^2}\left\| \widetilde{Q}(\xi) -\frac{\rho^2 }{d+2}I_d  \right \|+C_D\left(\frac{\eps}{\rho}+\rho^{\alpha}\right).
\]
From Lemma~\ref{lem:uph2} and a calculation similar to the one in Lemma~\ref{lem:qr2},  
\[
\left\|\widetilde{Q}(\xi)- \frac{\rho^{2}}{d+2}I_d \right\|\le \frac{C_Dm(O_{\xi} \bigtriangleup B_{+}(\hat{\xi},\rho))}{m(B_{+}(\hat{\xi},\rho))} \rho^2.
\]
The rest of the proof is similar to that of Lemma~\ref{lem:qr2}. 
\end{proof}

\section{Proofs of Theorems~\ref{thm:1} and \ref{thm:2}}

Define an operator  $U \colon \R^{\cK} \to \R^{\cK}$ by 
 \begin{align*}
Uf(\xi)&=\sum_{\eta \in \mathcal{N}_{\xi}}\left(f(\eta)- f(\xi) \right)\,\frac{m(\eta)}{m(O_{\xi})}. 
\end{align*}
Throughout this section,  we fix  $\xi \in \cK $  and $f \in C^2(\R^d)$.  
We first observe that 
\begin{align}\label{eq:U}
U(\pi f)(\xi)&= \sum_{\eta \in \mathcal{N}_{\xi}}(f(\bar{\eta}) - f(\bar{\xi})) \frac{ m(\eta)}{m(O_{\xi})} \\
&=\sum_{\eta \in \mathcal{N}_{\xi}}  (f(\bar{\eta}) -f(\check{\xi}))\frac{m(\eta)}{m(O_{\xi})} -(f(\bar{\xi})-f(\check{\xi})) \notag \\
&=:I_1-I_2. \notag 
\end{align}
For $a,b \in \R^d$ and $t \in [0,1]$,  define 
\[
u(t,a,b)=(1-t)\langle b-a, \nabla^2 f(a+t(b-a))(b-a)\rangle. \]
It then follows that
\begin{align} \label{eq:I134}
I_1
&= \sum_{\eta \in \mathcal{N}_{\xi}}  \langle \nabla f(\check{\xi}),\bar{\eta} -\check{\xi} \rangle   \frac{ m(\eta)}{m(O_{\xi})} + \sum_{\eta \in \mathcal{N}_{\xi}}\left( \int_{0}^1 u(t,\check{\xi},\bar{\eta})\,dt \right) \frac{m(\eta)  }{m(O_{\xi})}   \\
&=:I_3+I_4.\notag 
\end{align}

\begin{lemma}\label{lem:I321}
Assume \eqref{eq:A} and that $\xi$ belongs to $\cK \setminus  \partial \cK.$ Then,  
\begin{align*}
&I_3-\sum_{\eta \in \mathcal{N}_{\xi}} (\pi f(\eta)-\pi f(\xi))c(\xi,\eta)\frac{m(\eta)}{m(O_{\xi})}-I_2 \notag \\
&=-\sum_{\eta \in \mathcal{N}_{\xi}} \left(\int_{0}^1 u(t,\bar{\xi},\bar{\eta})\,dt \right) c(\xi,\eta)\frac{m(\eta)}{m(O_{\xi})}.
\end{align*}
\end{lemma}
\begin{proof}
Since $\xi \in \cK \setminus  \partial \cK$,  we have $I_2=0$.
From Taylor's formula,
\begin{align}
&\sum_{\eta \in \mathcal{N}_{\xi}} (\pi f(\eta)-\pi f(\xi))c(\xi,\eta)\frac{m(\eta)}{m(O_{\xi})}=\sum_{\eta \in \mathcal{N}_{\xi}}(f(\bar{\eta})-f(\bar{\xi}))c(\xi,\eta)\frac{m(\eta)}{m(O_{\xi})}  \label{eq:add} \\
&=\sum_{\eta \in \mathcal{N}_{\xi}} \langle \nabla f(\bar{\xi}), \bar{\eta}-\bar{\xi} \rangle c(\xi,\eta)\frac{m(\eta)}{m(O_{\xi})}+\sum_{\eta \in \mathcal{N}_{\xi}} \left(\int_{0}^1 u(t,\bar{\xi},\bar{\eta})\,dt \right) c(\xi,\eta)\frac{m(\eta)}{m(O_{\xi})} .  \notag
\end{align}
From the definition of $c(\xi)$ (see \eqref{eq:c0}), it follows that
\begin{align*}
&\sum_{\eta \in \mathcal{N}_{\xi}} \langle \nabla f(\bar{\xi}), \bar{\eta}-\bar{\xi} \rangle c(\xi,\eta)\frac{m(\eta)}{m(O_{\xi})}   \\
&=\langle \nabla f(\bar{\xi}), A (\xi)c(\xi) \rangle =\langle \nabla f(\bar{\xi}), A (\xi)A^{+}(\xi)b(\xi) \rangle.  \notag
\end{align*}
Lemma~\ref{lemma:span} implies that the rank of $\widetilde{A}(\xi)$ is equal to $d$.  Thus, the rank of $A(\xi)$ is also equal to $d$, and it holds that $A(\xi)A^{+}(\xi)b(\xi)=b(\xi)$. Hence, we have
\begin{align}
&\sum_{\eta \in \mathcal{N}_{\xi}} \langle \nabla f(\bar{\xi}), \bar{\eta}-\bar{\xi} \rangle c(\xi,\eta)\frac{m(\eta)}{m(O_{\xi})}  =\langle \nabla f(\bar{\xi}), b(\xi) \rangle=I_3.  \label{eq:add2}
\end{align}
We use \eqref{eq:add} and \eqref{eq:add2} to obtain the conclusion.
\end{proof}
For $a,b \in \R^d$ and $t \in [0,1]$,  define
\[
v(t,a,b)=(1-t)\langle b-a, (\nabla^2 f(a+t(b-a))-\nabla^2 f(a))(b-a)\rangle.
\]
 Then,  
 \begin{align}\label{eq:tayl}
u(t,a,b)&=v(t,a,b)+(1-t)\langle b-a, \nabla^2 f(a)(b-a) \rangle  \\
&=v(t,a,b)+(1-t)\text{\rm Tr}\left[\{(b-a)\otimes (b-a) \}\nabla^2 f(a)\right].\notag
\end{align}
\begin{lemma}\label{lem:I322}
Assume \eqref{eq:B} and that $\xi$ belongs to $ \partial \cK.$  Then,  it follows that
\begin{align*}
 &I_3-\sum_{\eta \in \mathcal{N}_{\xi}}( \pi f(\eta)-\pi f(\xi))c(\xi,\eta)\frac{m(\eta)}{m(O_{\xi})}-I_2\\
 &=-\sum_{\eta \in \mathcal{N}_{\xi}} \left(\int_{0}^1 u(t,\hat{\xi},\bar{\eta})\,dt \right)c(\xi,\eta) \frac{m(\eta)}{m(O_\xi)}\notag \\
&\quad  -\Lambda_0(\xi)\int_{0}^1 v(t,\hat{\xi},\bar{\xi})\,dt+\langle \nabla f(\hat{\xi}),\Lambda_1(\xi) \rangle-\frac{1}{2}\text{\rm Tr}\left[\Lambda_2(\xi)\nabla^2 f(\hat{\xi}) \right],
 \end{align*}
where we define $\Lambda_0(\xi)$, $\Lambda_1(\xi)$,  and $\Lambda_2(\xi)$ by
\begin{align*}
\Lambda_0(\xi)&= \sum_{\eta \in \mathcal{N}_{\xi}} \left(1-c(\xi,\eta)\right)\frac{m(\eta)}{m(O_{\xi})},\\
\Lambda_1(\xi)&=\beta_d \rho(\xi) \nu(\hat{\xi})- \Lambda_0(\xi) (\bar{\xi}-\hat{\xi}),\\
\Lambda_2(\xi)&=\Lambda_0(\xi)(\bar{\xi}-\hat{\xi})\otimes (\bar{\xi}-\hat{\xi}).
\end{align*}
\end{lemma}
\begin{proof}
Since $\xi \in \partial \cK$,  we have 
\begin{equation}\label{eq:w1}
I_3=\langle \nabla f(\hat{\xi}), b(\xi) + \beta_d \rho(\xi) \nu(\hat{\xi}) \rangle. 
\end{equation}
Using Taylor's formula,  we obtain that
\begin{align}\label{eq:w2}
&\sum_{\eta \in \mathcal{N}_{\xi}} (\pi f(\eta)-\pi f(\xi))c(\xi,\eta)\frac{m(\eta)}{m(O_{\xi})}  \\
&=\sum_{\eta \in \mathcal{N}_{\xi}}(f(\bar{\eta})-f(\hat{\xi}))c(\xi,\eta)\frac{m(\eta)}{m(O_{\xi})} -\sum_{\eta \in \mathcal{N}_{\xi}} (f(\bar{\xi})-f(\hat{\xi}))c(\xi,\eta)\frac{m(\eta)}{m(O_{\xi})}  \notag \\
&=\sum_{\eta \in \mathcal{N}_{\xi}} \langle \nabla f(\hat{\xi}),\bar{\eta}-\hat{\xi}\rangle c(\xi,\eta)\frac{m(\eta)}{m(O_{\xi})}+\sum_{\eta \in \mathcal{N}_{\xi}} \left(\int_{0}^1 u(t,\hat{\xi},\bar{\eta})\,dt \right) c(\xi,\eta) \frac{m(\eta)}{m(O_\xi)}  \notag \\
&\quad -\sum_{\eta \in \mathcal{N}_{\xi}} (f(\bar{\xi})-f(\hat{\xi}))c(\xi,\eta)\frac{m(\eta)}{m(O_{\xi})} .\notag
\end{align}
From Taylor's formula and \eqref{eq:tayl}, 
\begin{align}
I_2
=f(\bar{\xi})-f(\check{\xi})&=\langle \nabla f(\hat{\xi}), \bar{\xi}-\hat{\xi} \rangle+\int_{0}^1v(t,\hat{\eta},\bar{\xi})\,dt \label{eq:w3} \\
&\quad+\frac{1}{2} \text{\rm Tr}\left[\{(\bar{\xi}-\hat{\xi})\otimes (\bar{\xi}-\hat{\xi})\}\nabla^2 f(\hat{\xi}) \right]. \notag
\end{align}
Combining \eqref{eq:w1},  \eqref{eq:w2},  and \eqref{eq:w3},  we obtain that
\begin{align}
&I_3-\sum_{\eta \in \mathcal{N}_{\xi}} ( \pi f(\eta)-\pi f(\xi))c(\xi,\eta)\frac{m(\eta)}{m(O_{\xi})}-I_2  \label{eq:I3221c}
 \\
 &=\langle \nabla f(\hat{\xi}), b(\xi) +\beta_d \rho(\xi)\nu(\hat{\xi}) \rangle-\sum_{\eta \in \mathcal{N}_{\xi}} \langle \nabla f(\hat{\xi}),\bar{\eta}-\hat{\xi}\rangle c(\xi,\eta)\frac{m(\eta)}{m(O_{\xi})} \notag  \\
&\quad-\sum_{\eta \in \mathcal{N}_{\xi}} \left(\int_{0}^1 u(t,\hat{\xi},\bar{\eta})\,dt \right) c(\xi,\eta) \frac{m(\eta)}{m(O_\xi)}  -\Lambda_0(\xi)\langle \nabla f(\hat{\xi}), \bar{\xi}-\hat{\xi} \rangle \notag \\
&\quad -\Lambda_0(\xi) \int_{0}^1 v(t,\hat{\xi},\bar{\xi})\,dt-  \frac{1}{2} \text{\rm Tr}\left[\Lambda_2(\xi)\nabla^2 f(\hat{\xi}) \right]. \notag
\end{align}
Using \eqref{eq:c0},   we obtain that
\begin{align}
&\langle \nabla f(\hat{\xi}), b(\xi) +\beta_d \rho(\xi)\nu(\hat{\xi}) \rangle-\sum_{\eta \in \mathcal{N}_{\xi}} \langle \nabla f(\hat{\xi}),\bar{\eta}-\hat{\xi}\rangle c(\xi,\eta)\frac{m(\eta)}{m(O_{\xi})}  
\label{eq:I3221d}  \\
&\quad -  \Lambda_0(\xi)\langle \nabla f(\hat{\xi}), \bar{\xi}-\hat{\xi} \rangle  \notag  \\
&=\langle \nabla f(\hat{\xi}), b(\xi) \rangle
-\sum_{\eta \in \mathcal{N}_{\xi}} \langle \nabla f(\hat{\xi}),\bar{\eta}-\hat{\xi}\rangle c(\xi,\eta)\frac{m(\eta)}{m(O_{\xi})}  \notag  \\
&\quad +\langle \nabla f(\hat{\xi}), \beta_d \rho(\xi)\nu(\hat{\xi})- \Lambda_0(\xi) (\bar{\xi}-\hat{\xi}) \rangle \notag \\
&=\langle \nabla f(\hat{\xi}), b(\xi) \rangle
-\sum_{\eta \in \mathcal{N}_{\xi}} \langle \nabla f(\hat{\xi}),\bar{\eta}-\hat{\xi}\rangle c(\xi,\eta)\frac{m(\eta)}{m(O_{\xi})}  \notag  \\
&\quad +\langle \nabla f(\hat{\xi}), \beta_d \rho(\xi)\nu(\hat{\xi})- \Lambda_0(\xi) (\bar{\xi}-\hat{\xi}) \rangle \notag\\
&=\langle \nabla f(\hat{\xi}), b(\xi) \rangle
-\sum_{\eta \in \mathcal{N}_{\xi}} \langle \nabla f(\hat{\xi}),\bar{\eta}-\hat{\xi}\rangle c(\xi,\eta)\frac{m(\eta)}{m(O_{\xi})}   \notag \\
&\quad+\langle \nabla f(\hat{\xi}),  \Lambda_1(\xi) \rangle.\notag 
\end{align}
From Lemma~\ref{lem:span22} and an argument similar to that in Lemma~\ref{lem:I321},   it follows that
\begin{align}
&\langle \nabla f(\hat{\xi}), b(\xi) \rangle
-\sum_{\eta \in \mathcal{N}_{\xi}} \langle \nabla f(\hat{\xi}),\bar{\eta}-\hat{\xi}\rangle c(\xi,\eta)\frac{m(\eta)}{m(O_{\xi})} \label{eq:I3221dd} \\
&=\langle \nabla f(\hat{\xi}), b(\xi) \rangle
-\langle \nabla f(\hat{\xi}),A(\xi)A^{+}(\xi)b(\xi)\rangle =0. \notag
\end{align}
Combining  \eqref{eq:I3221c}, \eqref{eq:I3221d},  and \eqref{eq:I3221dd}, we obtain the claim. 
\end{proof}
\begin{lemma}\label{lem:I4}
We have 
\begin{align*}
I_4&= \sum_{\eta \in \mathcal{N}_{\xi}}\left( \int_{0}^1 v(t,\check{\xi},\bar{\eta}) \,dt \right) \frac{m(\eta)}{m(O_{\xi})} -\frac{1}{2} \sum_{\eta \in \mathcal{N}_{\xi}} \left( \fint_{\eta } u(0,\bar{\eta},y)\,m(dy) \right)\frac{m(\eta)}{m(O_{\xi})} \\ 
 &\quad+\frac{1}{2} \sum_{\eta \in \mathcal{N}_{\xi}} \left( \fint_{\eta } \langle y-\check{\xi}, (\nabla^2 f(\check{\xi})-\nabla^2 f(\bar{\eta}))(\bar{\eta}-y) \rangle\,m(dy) \right)\frac{m(\eta)}{m(O_{\xi})} \notag  \\
&\quad +\frac{1}{2}\text{\rm Tr}\left[ \widetilde{Q}(\xi)\nabla^2 f(\check{\xi})\right] .
\end{align*}
\end{lemma}

\begin{proof}
It follows that  
\begin{align}
I_4&= \sum_{\eta \in \mathcal{N}_{\xi}}\left[ \int_{0}^1\{ u(t,\check{\xi},\bar{\eta}) -(1-t)u(0,\check{\xi},\bar{\eta}) \}\,dt \right] \frac{m(\eta)}{m(O_{\xi})}  \notag  \\
&\quad + \sum_{\eta \in \mathcal{N}_{\xi}}\left( \int_{0}^1 (1-t)u(0,\check{\xi},\bar{\eta})\,dt -\frac{1}{2} \fint_{\eta}\langle y-\check{\xi},\nabla^2f(\check{\xi})(y-\check{\xi})\rangle \,m(dy)  \right)\frac{m(\eta) }{m(O_{\xi})} \notag  \\
&\quad +\frac{1}{2} \sum_{\eta \in \mathcal{N}_{\xi}} \left( \fint_{\eta}\langle y-\check{\xi},\nabla^2f(\check{\xi})(y-\check{\xi})\rangle \,m(dy)\right) \frac{m(\eta) }{m(O_{\xi})} \notag \\
&=:I_5+I_6+I_7.  \notag
\end{align}
A direct computation shows  that for every $\eta \in \mathcal{N}_{\xi}$,
\begin{align*}
&\int_{0}^1\{ u(t,\check{\xi},\bar{\eta}) -(1-t)u(0,\check{\xi},\bar{\eta}) \}\,dt  \\
&=\int_{0}^1 (1-t)\{ \langle \bar{\eta}-\check{\xi}, \nabla^2 f(\check{\xi}+t(\bar{\eta}-\check{\xi}))(\bar{\eta}-\check{\xi}) \rangle -\langle \bar{\eta}-\check{\xi}, \nabla^2 f(\check{\xi})(\bar{\eta}-\check{\xi})\rangle \}\,dt \\
&=\int_{0}^1 v(t,\check{\xi},\bar{\eta}) \,dt.
\end{align*}
Therefore, 
\begin{equation}\label{eq:I5}
I_5=\sum_{\eta \in \mathcal{N}_{\xi}}\left( \int_{0}^1 v(t,\check{\xi},\bar{\eta}) \,dt \right) \frac{m(\eta)}{m(O_{\xi})}.
\end{equation}
We obtain that 
\begin{align}
I_6=\frac{1}{2}\sum_{\eta \in \mathcal{N}_{\xi}}\left[  \fint_{\eta} \{u(0,\check{\xi},\bar{\eta})-u(0,\check{\xi},y)\}\,m(dy)  \right] \frac{m(\eta) }{m(O_{\xi})} . \label{eq:i60}
\end{align}
For $\eta \in \mathcal{N}_{\xi}$ and $y \in \eta$,
\begin{align}\label{eq:i61}
&u(0,\check{\xi},\bar{\eta})-u(0,\check{\xi},y)  \\
&=u(0,\check{\xi},\bar{\eta})-u(0,\bar{\eta},y)+u(0,\bar{\eta},y)-u(0,\check{\xi},y).  \notag
\end{align}
The sum of all terms except the second term in the above equation is
\begin{align}\label{eq:i62}
&u(0,\check{\xi},\bar{\eta})+u(0,\bar{\eta},y)-u(0,\check{\xi},y)   \\
&=\langle \bar{\eta}-\check{\xi}, \nabla^2 f(\check{\xi})(\bar{\eta}-\check{\xi})\rangle  \notag \\
&\quad +\langle y-\bar{\eta}, \nabla^2 f(\bar{\eta})(y-\bar{\eta})\rangle-\langle y-\check{\xi},\nabla^2f(\check{\xi})(y-\check{\xi})\rangle. \notag
\end{align}
The difference  of the first and third terms of the above equation  is equal to 
\begin{align}
&\langle \bar{\eta}-\check{\xi}, \nabla^2 f(\check{\xi})(\bar{\eta}-\check{\xi}) \rangle-\langle y-\check{\xi}, \nabla^2 f(\check{\xi})(y-\check{\xi}) \rangle 
\label{eq:i63}
\\
&=\langle \bar{\eta}-\check{\xi}, \nabla^2 f(\check{\xi})(\bar{\eta}-\check{\xi}) \rangle-\langle y-\check{\xi}, \nabla^2 f(\check{\xi})(\bar{\eta}-\check{\xi}) \rangle \notag\\
&\quad +\langle y-\check{\xi}, \nabla^2 f(\check{\xi})(\bar{\eta}-\check{\xi}) \rangle-\langle y-\check{\xi}, \nabla^2 f(\check{\xi})(y-\check{\xi}) \rangle  \notag \\
&=\langle \bar{\eta}-y, \nabla^2 f(\check{\xi})(\bar{\eta}-\check{\xi}) \rangle+\langle y-\check{\xi}, \nabla^2 f(\check{\xi})(\bar{\eta}-y) \rangle. \notag 
\end{align}
From \eqref{eq:i62} and \eqref{eq:i63},  we have
\begin{align}\label{eq:i64}
&u(0,\check{\xi},\bar{\eta})+u(0,\bar{\eta},y)-u(0,\check{\xi},y)  \\
&=\langle \bar{\eta}-y, \nabla^2 f(\check{\xi})(\bar{\eta}-\check{\xi}) \rangle+\langle y-\check{\xi}, \nabla^2 f(\check{\xi})(\bar{\eta}-y) \rangle \notag \\
&\quad+\langle y-\bar{\eta}, \nabla^2 f(\bar{\eta})(y-\bar{\eta}) \rangle  \notag \\
&=\langle \bar{\eta}-y, \nabla^2 f(\check{\xi})(\bar{\eta}-\check{\xi}) \rangle+\langle y-\check{\xi}, \nabla^2 f(\check{\xi})(\bar{\eta}-y) \rangle  \notag \\
&\quad+\langle y-\check{\xi}, \nabla^2 f(\bar{\eta})(y-\bar{\eta}) \rangle+\langle \check{\xi}-\bar{\eta}, \nabla^2 f(\bar{\eta})(y-\bar{\eta}) \rangle \notag  \\
&=\langle \bar{\eta}-y, \nabla^2 f(\check{\xi})(\bar{\eta}-\check{\xi}) \rangle+\langle y-\check{\xi}, (\nabla^2 f(\check{\xi})-\nabla^2 f(\bar{\eta}))(\bar{\eta}-y) \rangle \notag  \\
&\quad +\langle \check{\xi}-\bar{\eta}, \nabla^2 f(\bar{\eta})(y-\bar{\eta}) \rangle  .\notag
\end{align}
From \eqref{eq:i60}, \eqref{eq:i61}, \eqref{eq:i64}, and the fact that $\bar{\eta}$ is the center of gravity of $\eta$,  it follows that
\begin{align*}
I_6
 &=-\frac{1}{2} \sum_{\eta \in \mathcal{N}_{\xi}} \left( \fint_{\eta } u(0,\bar{\eta},y)\,m(dy) \right)\frac{m(\eta)}{m(O_{\xi})} \\ 
 &\quad+\frac{1}{2} \sum_{\eta \in \mathcal{N}_{\xi}} \left( \fint_{\eta } \langle y-\check{\xi}, (\nabla^2 f(\check{\xi})-\nabla^2 f(\bar{\eta}))(\bar{\eta}-y) \rangle\,m(dy) \right)\frac{m(\eta)}{m(O_{\xi})}.
\end{align*}
It is clear  that 
\begin{align*}
I_7&=\frac{1}{2}  \fint_{O_\xi}\langle y-\check{\xi},\nabla^2f(\check{\xi})(y-\check{\xi})\rangle \,m(dy)  \notag \\
&=\frac{1}{2}\text{\rm Tr}\left[ \left(\fint_{O_{\xi}}(y-\check{\xi})\otimes (y-\check{\xi})\,m(dy) \right)\nabla^2 f(\check{\xi})\right]=\frac{1}{2}\text{\rm Tr}\left[ \widetilde{Q}(\xi)\nabla^2 f(\check{\xi})\right].
\end{align*}
Hence,  we obtain the desired equation.  
\end{proof}

Combining \eqref{eq:U},  \eqref{eq:I134},  and Lemmas~\ref{lem:I321} and \ref{lem:I4},  we obtain the following lemma.

\begin{lemma}\label{lem:INT}
Assume \eqref{eq:A} and that $\xi$ belongs to $\cK \setminus  \partial \cK.$  Then,  we have 
\begin{align}\label{eq:INT}
&U(\pi f)(\xi)-\sum_{\eta \in \mathcal{N}_{\xi}} (\pi f(\eta)-\pi f(\xi))c(\xi,\eta)\frac{m(\eta)}{m(O_{\xi})}   \\
&=-\sum_{\eta \in \mathcal{N}_{\xi}} \left(\int_{0}^1 u(t,\bar{\xi},\bar{\eta})\,dt \right) c(\xi,\eta)\frac{m(\eta)}{m(O_{\xi})} \notag \\
&\quad +\sum_{\eta \in \mathcal{N}_{\xi}}\left( \int_{0}^1 v(t,\check{\xi},\bar{\eta}) \,dt \right) \frac{m(\eta)}{m(O_{\xi})}\notag \\ 
&\quad -\frac{1}{2} \sum_{\eta \in \mathcal{N}_{\xi}} \left( \fint_{\eta } u(0,\bar{\eta},y)\,m(dy) \right)\frac{m(\eta)}{m(O_{\xi})} \notag \\ 
 &\quad+\frac{1}{2} \sum_{\eta \in \mathcal{N}_{\xi}} \left( \fint_{\eta } \langle y-\check{\xi}, (\nabla^2 f(\check{\xi})-\nabla^2 f(\bar{\eta}))(\bar{\eta}-y) \rangle\,m(dy) \right)\frac{m(\eta)}{m(O_{\xi})} \notag  \\
&\quad +\frac{1}{2}\text{\rm Tr}\left[ \widetilde{Q}(\xi)\nabla^2 f(\check{\xi})\right] . \notag
\end{align}
\end{lemma}

By using  \eqref{eq:U},  \eqref{eq:I134}, and Lemmas~\ref{lem:I322} and \ref{lem:I4},   we also get the following.
\begin{lemma}\label{lem:BD}
Assume \eqref{eq:B} and that $\xi$ belongs to $ \partial \cK.$  Then,  we have 
\begin{align}\label{eq:BD}
&U(\pi f)(\xi)-\sum_{\eta \in \mathcal{N}_{\xi}} (\pi f(\eta)-\pi f(\xi))c(\xi,\eta)\frac{m(\eta)}{m(O_{\xi})}  \\
&=-\sum_{\eta \in \mathcal{N}_{\xi}} \left(\int_{0}^1 u(t,\hat{\xi},\bar{\eta})\,dt \right)c(\xi,\eta) \frac{m(\eta)}{m(O_\xi)}\notag \\
&\quad  -\Lambda_0(\xi)\int_{0}^1 v(t,\hat{\xi},\bar{\xi})\,dt+\langle \nabla f(\hat{\xi}),\Lambda_1(\xi) \rangle-\frac{1}{2}\text{\rm Tr}\left[\Lambda_2(\xi)\nabla^2 f(\hat{\xi}) \right]  \notag \\
&\quad + \sum_{\eta \in \mathcal{N}_{\xi}}\left( \int_{0}^1 v(t,\hat{\xi},\bar{\eta}) \,dt \right) \frac{m(\eta)}{m(O_{\xi})} \notag \\
&\quad-\frac{1}{2} \sum_{\eta \in \mathcal{N}_{\xi}} \left( \fint_{\eta } u(0,\bar{\eta},y)\,m(dy) \right)\frac{m(\eta)}{m(O_{\xi})} \notag  \\ 
 &\quad+\frac{1}{2} \sum_{\eta \in \mathcal{N}_{\xi}} \left( \fint_{\eta } \langle y-\hat{\xi}, (\nabla^2 f(\hat{\xi})-\nabla^2 f(\bar{\eta}))(\bar{\eta}-y) \rangle\,m(dy) \right)\frac{m(\eta)}{m(O_{\xi})} \notag  \\
&\quad +\frac{1}{2}\text{\rm Tr}\left[ \widetilde{Q}(\xi)\nabla^2 f(\hat{\xi})\right] . \notag
\end{align}
\end{lemma}

We estimate the right-hand sides of \eqref{eq:INT} and \eqref{eq:BD}.

\begin{lemma}\label{lem:I321b}
Assume \eqref{eq:A} and that $\xi$ belongs to $\cK \setminus  \partial \cK.$ Then, there exists $C>0$ such that
\begin{align*}
&\left|\sum_{\eta \in \mathcal{N}_{\xi}} \left(\int_{0}^1 u(t,\bar{\xi},\bar{\eta})\,dt \right) c(\xi,\eta)\frac{m(\eta)}{m(O_{\xi})} \right| \le C\eps(\xi)\rho(\xi)\sup_{y \in B(\bar{\xi},\rho(\xi))}\|\nabla^2 f(y)\|.
\end{align*}
\end{lemma}
\begin{proof}
For any $\eta \in \mathcal{N}_\xi$, we have $|\bar{\eta}-\bar{\xi}| < \rho(\xi)$. Therefore,  for any $t \in [0,1]$, 
\begin{align*}
\max_{\eta \in \mathcal{N}_\xi}|u(t,\bar{\xi},\bar{\eta})| &=\max_{\eta \in \mathcal{N}_\xi}|(1-t)\langle \bar{\eta}-\bar{\xi}, \nabla^2 f(\bar{\xi}+t(\bar{\eta}-\bar{\xi}))(\bar{\eta}-\bar{\xi})\rangle| \notag \\
&\le(1-t) \rho(\xi)^2 \sup_{y \in B(\bar{\xi},\rho(\xi))}\|\nabla^2 f(y)\|.
\end{align*}
Using this and Lemma~\ref{lem:A},  we obtain the claim.
\end{proof}

\begin{lemma}\label{lem:I322b}
Assume  that $D$ is a $C^{1,\alpha}$-domain for some $\alpha  \in (0,1]$. Also assume that  \eqref{eq:B} holds.   Let $\xi$ belong to $  \partial \cK$.  Then,  there exists $C_D>0$ such that
\begin{align*}
&\left|\sum_{\eta \in \mathcal{N}_{\xi}} \left(\int_{0}^1 u(t,\hat{\xi},\bar{\eta})\,dt \right)c(\xi,\eta) \frac{m(\eta)}{m(O_\xi)}+\Lambda_0(\xi)\int_{0}^1 v(t,\hat{\xi},\bar{\xi})\,dt \right| \notag \\
&\le C_D\left(\frac{ \eps(\xi)}{\rho(\xi)}+\rho(\xi)^\alpha \right) \rho(\xi)^2\sup_{y \in B(\hat{\xi},\rho(\xi))}\|\nabla^2 f(y)\| \\
&\quad +C_D \dl(\xi)^2 \sup_{x \in B(\hat{\xi},\dl(\xi))} \|\nabla^2 f(x)-\nabla^2 f(\hat{\xi}) \|.
\end{align*}
\end{lemma}
\begin{proof}
We have $|\bar{\eta}-\hat{\xi}| < \rho(\xi)$ for any $\eta \in \mathcal{N}_\xi$.  Thus,  for any $t \in [0,1]$, 
\begin{align}\label{eq:w4}
\max_{\eta \in \mathcal{N}_\xi}|u(t,\hat{\xi},\bar{\eta})| &=\max_{\eta \in \mathcal{N}_\xi}|(1-t)\langle \bar{\eta}-\hat{\xi}, \nabla^2 f(\hat{\xi}+t(\bar{\eta}-\hat{\xi}))(\bar{\eta}-\hat{\xi})\rangle|  \\
&\le (1-t)\rho(\xi)^2 \sup_{y \in B(\hat{\xi},\rho(\xi))}\|\nabla^2 f(y)\|.  \notag
\end{align}
By using Lemma~\ref{lem:ed}~(2),  we have for any $t \in [0,1]$, 
\begin{align*}
|v(t,\hat{\xi},\bar{\xi})|&=(1-t)|\langle \bar{\xi}-\hat{\xi}, (\nabla^2 f(\hat{\xi}+t(\bar{\xi}-\hat{\xi}))-\nabla^2 f(\hat{\xi}))(\bar{\xi}-\hat{\xi})\rangle| \\
&\le  (1-t) \dl(\xi)^2 \sup_{x \in B(\hat{\xi},\dl(\xi))} \|\nabla^2 f(x)-\nabla^2 f(\hat{\xi}) \|.
\end{align*}
Therefore,  
\begin{align}
&\left| \Lambda_0(\xi)\int_{0}^1 v(t,\hat{\xi},\bar{\xi})\,dt \right| \label{eq:w5} \\
&\le \frac{1}{2}\left( 1+\max_{\eta \in \mathcal{N}_{\xi}} |c(\xi,\eta)|\right) \dl(\xi)^2 \sup_{x \in B(\hat{\xi},\dl(\xi))} \|\nabla^2 f(x)-\nabla^2 f(\hat{\xi}) \|.  \notag 
\end{align}
Combining \eqref{eq:w4},  \eqref{eq:w5},  and Lemma~\ref{lem:B},  we obtain the conclusion. 
\end{proof}

\begin{lemma}\label{lem:I4b}
It follows that 
\begin{align*}
&\left|  \sum_{\eta \in \mathcal{N}_{\xi}}\left( \int_{0}^1 v(t,\check{\xi},\bar{\eta}) \,dt \right) \frac{m(\eta)}{m(O_{\xi})} -\frac{1}{2} \sum_{\eta \in \mathcal{N}_{\xi}} \left( \fint_{\eta } u(0,\bar{\eta},y)\,m(dy) \right)\frac{m(\eta)}{m(O_{\xi})} \right.  \\ 
 &\quad \left. +\frac{1}{2} \sum_{\eta \in \mathcal{N}_{\xi}} \left( \fint_{\eta } \langle y-\check{\xi}, (\nabla^2 f(\check{\xi})-\nabla^2 f(\bar{\eta}))(\bar{\eta}-y) \rangle\,m(dy) \right)\frac{m(\eta)}{m(O_{\xi})} \right| \\
&\le \left(\rho(\xi)^2+\eps(\xi)^2+\eps(\xi)\rho(\xi)  \right)\sup_{y \in B(\check{\xi},\rho(\xi))} \|\nabla^2 f(y)-\nabla^2 f(\check{\xi})\|\\
&\quad+\frac{\eps(\xi)^2}{2}  \max_{\eta \in \mathcal{N}_\xi}\|\nabla^2 f(\bar{\eta}) \| .
\end{align*}
\end{lemma}
\begin{proof}
For any $ \eta \in \mathcal{N}_{\xi}$, we have 
$|\bar{\eta}-\check{\xi}| < \rho(\xi)$. Therefore,  we have for any  $\eta \in \mathcal{N}_{\xi}$ and $t \in [0,1]$, 
\begin{align*}
|v(t,\check{\xi},\bar{\eta})| &\le (1-t) \|\nabla^2 f(\check{\xi}+t(\bar{\eta}-\check{\xi}))-\nabla^2 f(\check{\xi}) \| |\bar{\eta}-\check{\xi}|^2  \notag \\
&\le (1-t)  \rho(\xi)^2 \sup_{y \in B(\check{\xi},\rho(\xi))} \|\nabla^2 f(y)-\nabla^2 f(\check{\xi})\|.
\end{align*}
This implies that
\begin{align}\label{eq:I4b1}
&\left|\sum_{\eta \in \mathcal{N}_{\xi}}\left( \int_{0}^1 v(t,\check{\xi},\bar{\eta}) \,dt \right) \frac{m(\eta)}{m(O_{\xi})} \right|  \le \frac{\rho(\xi)^2}{2} \sup_{y \in B(\check{\xi},\rho(\xi))} \|\nabla^2 f(y)-\nabla^2 f(\check{\xi})\|. 
\end{align}
From Lemma~\ref{lem:ed}~(1),  we have  $|y-\bar{\eta}| \le \eps(\xi)$ for any $ \eta \in \mathcal{N}_{\xi}$ and $y \in \eta$.   Therefore, 
\begin{align}\label{eq:I4b2}
&\left| \frac{1}{2} \sum_{\eta \in \mathcal{N}_{\xi}} \left( \fint_{\eta } u(0,\bar{\eta},y)\,m(dy) \right)\frac{m(\eta)}{m(O_{\xi})} \right| \le  \frac{\eps(\xi)^2}{2}  \max_{\eta \in \mathcal{N}_\xi}\|\nabla^2 f(\bar{\eta}) \|.
\end{align}
We also use Lemma~\ref{lem:ed}~(1) to see that for any $ \eta \in \mathcal{N}_{\xi}$ and $y \in \eta$,
\[|y-\check{\xi}| \le |y-\bar{\eta}|+|\bar{\eta}-\check{\xi}| \le \eps(\xi)+\rho(\xi) 
\]
and
\begin{align}\label{eq:I4b3}
&|\langle y-\check{\xi}, (\nabla^2 f(\check{\xi})-\nabla^2 f(\bar{\eta}))(\bar{\eta}-y) \rangle|  \\
& \le  \eps(\xi)(\eps(\xi)+\rho(\xi))   \sup_{y \in B(\check{\xi},\rho(\xi))} \|\nabla^2 f(y)-\nabla^2 f(\check{\xi})\| . \notag
\end{align}
Combining  \eqref{eq:I4b1},  \eqref{eq:I4b2},  and  \eqref{eq:I4b3},  we complete the proof. 
\end{proof}

\begin{lemma}\label{lem:QL}
If $\xi \in \cK \setminus \partial \cK$,  we have 
\begin{align}\label{eq:QL1}
&\left|\text{\rm Tr}\left[ \frac{\widetilde{Q}(\xi)}{q(\xi)}\nabla^2 f(\bar{\xi})\right]-\pi \left( \Delta f\right)(\xi)\right| \le d\left\| \frac{\widetilde{Q}(\xi)}{q(\xi)}-I_d \right\| \| \nabla^2 f(\bar{\xi})\| .
\end{align}
If  $\xi \in \partial \cK$,  we have 
\begin{align}
&\left| \text{\rm Tr}\left[ \frac{\widetilde{Q}(\xi)}{q(\xi)}\nabla^2 f(\hat{\xi})\right]-\pi \left( \Delta f\right)(\xi) \right|  \label{eq:QL2}  \\
&\le d\left\| \frac{\widetilde{Q}(\xi)}{q(\xi)}-I_d\right\| \| \nabla^2 f(\hat{\xi})\|  +\sup_{x \in B(\bar{\xi},\dl(\xi))}|\Delta f(x)-\Delta f(\bar{\xi})|.\notag
\end{align}
\end{lemma}

\begin{proof}
From the definition of $\pi$,  it follows that 
\begin{align}\label{eq:QL3}
&\left| \text{\rm Tr}\left[ \frac{\widetilde{Q}(\xi)}{q(\xi)}\nabla^2 f(\check{\xi})\right]-\pi \left( \Delta f\right)(\xi) \right|   \\
&=\left| \text{\rm Tr}\left[ \left(\frac{\widetilde{Q}(\xi)}{q(\xi)}-I_d \right)\nabla^2 f(\check{\xi})\right]+ \left(\Delta f(\check{\xi})- \Delta f(\bar{\xi}) \right) \right| \notag \\
&\le d \left\| \frac{\widetilde{Q}(\xi)}{q(\xi)}-I_d\right\| \| \nabla^2 f(\check{\xi})\| +|\Delta f(\check{\xi})- \Delta f(\bar{\xi})| .  \notag
\end{align}
If $\xi \in \cK \setminus  \partial \cK$,  the second term of the right-hand side of the above equation is equal to zero.  Thus,   \eqref{eq:QL1} follows.  If $\xi \in \partial \cK$,  Lemma~\ref{lem:ed}~(2) implies that
\begin{align}
 |\Delta f(\hat{\xi})- \Delta f(\bar{\xi}) | \le \sup_{x \in B(\hat{\xi},\dl(\xi))}|\Delta f(x)-\Delta f(\hat{\xi})|. \label{eq:QL4}
\end{align}
 We obtain \eqref{eq:QL2} from \eqref{eq:QL3} and \eqref{eq:QL4}.
\end{proof}

We now prove Theorems~\ref{thm:1} and \ref{thm:2}.

\begin{proof}[Proof of Theorem~\ref{thm:1}]
Using Lemmas~\ref{lem:INT},  \ref{lem:I321b},  and \ref{lem:I4b},  we see that there exists $C>0$  such that 
\begin{align*}
&\left| U(\pi f)(\xi)-\sum_{\eta \in \mathcal{N}_{\xi}} (\pi f(\eta)-\pi f(\xi))c(\xi,\eta)\frac{m(\eta)}{m(O_{\xi})} - \frac{1}{2}\text{\rm Tr}\left[ \widetilde{Q}(\xi)\nabla^2 f(\bar{\xi})\right] \right|\\
&\le C\eps(\xi)\rho(\xi) \sup_{y \in B(\bar{\xi},\rho(\xi))}\|\nabla^2 f(y)\|\\ 
 &\quad+  \left(\rho(\xi)^2+\eps(\xi)^2+\eps(\xi)\rho(\xi)  \right)\sup_{y \in B(\bar{\xi},\rho(\xi))} \|\nabla^2 f(y)-\nabla^2 f(\bar{\xi})\|\\
&\quad+\frac{\eps(\xi)^2}{2}  \max_{\eta \in \mathcal{N}_\xi}\|\nabla^2 f(\bar{\eta}) \| \\
&\le  C\left(\eps(\xi)\rho(\xi)+\frac{\eps(\xi)^2}{2}   \right)\sup_{y \in B(\bar{\xi},\rho(\xi))}\|\nabla^2 f(y)\|\\
&\quad+ \left(\rho(\xi)^2+\eps(\xi)^2+\eps(\xi)\rho(\xi)  \right) \sup_{y \in B(\bar{\xi},\rho(\xi))} \|\nabla^2 f(y)-\nabla^2 f(\bar{\xi})\|.
\end{align*}
In the last inequality,  we used the fact that $|\bar{\eta}-\bar{\xi}|<\rho(\xi)$ for $\eta \in \mathcal{N}_\xi$.  Dividing both sides of the above inequality by $q(\xi)$ and applying \eqref{eq:A} and the third inequality of Lemma~\ref{lem:A} to the resulting one,  we see that there exists $C>0$ such that
\begin{align*}
&\left| L(\pi f)(\xi)- \frac{1}{2}\text{\rm Tr}\left[ \frac{\widetilde{Q}(\xi)}{q(\xi)}\nabla^2 f(\bar{\xi})\right] \right|\\
&\le \frac{C\eps(\xi)}{\rho(\xi)} \sup_{y \in B(\bar{\xi},\rho(\xi))}\|\nabla^2 f(y)\| +C \sup_{y \in B(\bar{\xi},\rho(\xi))} \|\nabla^2 f(y)-\nabla^2 f(\bar{\xi})\|.
\end{align*}
Applying Lemmas~\ref{lem:Q/q} and \ref{lem:QL} to  the left-hand side of the above inequality,  we complete the proof.
\end{proof}

\begin{proof}[Proof of Theorem~\ref{thm:2}]
Applying  Lemmas~\ref{lem:I322b} and \ref{lem:I4b}  to \eqref{eq:BD},  we see that  there exists $C_D>0$ such that 
\begin{align*}
&\Biggl| U(\pi f)(\xi)-\sum_{\eta \in \mathcal{N}_{\xi}} (\pi f(\eta)-\pi f(\xi))c(\xi,\eta)\frac{m(\eta)}{m(O_{\xi})} - \frac{1}{2}\text{\rm Tr}\left[ \widetilde{Q}(\xi)\nabla^2 f(\hat{\xi})\right]     \\
&-\langle \nabla f(\hat{\xi}),\Lambda_1(\xi) \rangle+\frac{1}{2}\text{\rm Tr}\left[\Lambda_2(\xi)\nabla^2 f(\hat{\xi}) \right]  \Biggr|\\
&\le C_D \rho(\xi)^2\left(\frac{ \eps(\xi)}{\rho(\xi)}+\rho(\xi)^\alpha \right) \sup_{y \in B(\hat{\xi},\rho(\xi))}\|\nabla^2 f(y)\| \\
&\quad +C_D \dl(\xi)^2 \sup_{x \in B(\hat{\xi},\dl(\xi))} \|\nabla^2 f(x)-\nabla^2 f(\hat{\xi})) \|\\
&\quad +  \left(\rho(\xi)^2+\eps(\xi)^2+\eps(\xi)\rho(\xi)  \right) \sup_{y \in B(\hat{\xi},\rho(\xi))} \|\nabla^2 f(y)-\nabla^2 f(\hat{\xi})\|\\
&\quad+\frac{\eps(\xi)^2}{2}  \max_{\eta \in \mathcal{N}_\xi}\|\nabla^2 f(\bar{\eta}) \|.
\end{align*} 
Divide both sides of the above inequality by $q(\xi)$.  Then,  applying \eqref{eq:B} and the third inequality of Lemma~\ref{lem:B} to the resulting one,  and using the facts that $\rho(\xi)>\dl(\xi)$ and $|\bar{\eta}-\hat{\xi}|<\rho(\xi)$ for $\eta \in \mathcal{N}_\xi$,   we see that there exists $C_D>0$ such that
\begin{align*}
&\left| L(\pi f)(\xi) - \frac{1}{2}\text{\rm Tr}\left[ \frac{Q(\xi)}{q(\xi)} \nabla^2 f(\hat{\xi})\right] -\frac{1}{q(\xi)} \left( \langle \nabla f(\hat{\xi}),\Lambda_1(\xi) \rangle -\frac{1}{2}\text{\rm Tr}\left[\Lambda_2(\xi)\nabla^2 f(\hat{\xi}) \right]\right) \right| \notag \\
&\le C_D \left(\frac{ \eps(\xi)}{\rho(\xi)}+\rho(\xi)^\alpha \right) \sup_{y \in B(\hat{\xi},\rho(\xi))}\|\nabla^2 f(y)\|  \notag \\
&\quad +C_D \sup_{y \in B(\hat{\xi},\rho(\xi))} \|\nabla^2 f(y)-\nabla^2 f(\hat{\xi}) \|. 
\end{align*} 
Applying Lemmas~\ref{lem:Q/q2}  and \ref{lem:QL} to the left-hand side of the above inequality,  we obtain  
\begin{align}
&\left| L(\pi f)(\xi) -\pi \left(\frac{\Delta f}{2}\right) (\xi)-\frac{1}{q(\xi)} \left( \langle \nabla f(\hat{\xi}),\Lambda_1(\xi) \rangle -\frac{1}{2}\text{\rm Tr}\left[\Lambda_2(\xi)\nabla^2 f(\hat{\xi}) \right]\right) \right| \label{eq:LN0} \\
&\le C_D \left(\frac{ \eps(\xi)}{\rho(\xi)}+\rho(\xi)^\alpha \right) \sup_{y \in B(\hat{\xi},\rho(\xi))}\|\nabla^2 f(y)\| \notag \\
&\quad +C_D \sup_{y \in B(\hat{\xi},\rho(\xi))} \|\nabla^2 f(y)-\nabla^2 f(\hat{\xi}) \| \notag 
\end{align} 
for some $C_D>0$. We now assume that $f \in C_{\text{Neu}}^2(\R^d)$.   From the definition of  $\hat{\xi}$,  the vector $\Lambda_1(\xi)$ is a constant multiple of $\nu(\hat{\xi})$.  
Because $f$ satisfies $\langle \nabla f(x),\nu(x)\rangle =0$ for every $x \in \partial D$,   we have
\begin{equation}\label{eq:LN2}
\langle \nabla f(\hat{\xi}),\Lambda_1(\xi) \rangle =0.
\end{equation}
The fact that $|\bar{\xi}-\hat{\xi}|< \dl(\xi)$ and the third inequality of Lemma~\ref{lem:B} imply that there exists $C_D>0$ such that
\begin{equation}\label{eq:LN3}
\left|\frac{ \Lambda_2(\xi)}{q(\xi)} \right| \le \frac{C_D\dl(\xi)}{\rho(\xi)}.
\end{equation}
Applying \eqref{eq:LN2}  and \eqref{eq:LN3} to the left-hand side of \eqref{eq:LN0},  we obtain the conclusion.
\end{proof}

\section{Proof of Theorem~\ref{thm:4}}
 Let $n \in \N$.  As explained in Section~2,  we may assume that $L^{(n)}$ is a bounded linear operator on $\mathcal{B}_b(\cK^{(n)})$ and  generates a continuous-time Markov chain $X^{(n)}=(\{X_t^{(n)}\}_{t \ge 0},\{P_\xi^{(n)}\}_{\xi \in \cK^{(n)}})$ on $\cK^{(n)}$.   Hence,  we have \eqref{eq:cflgn}.  Recall that the semigroup $\{p_t^{(n)}\}_{t> 0}$ associated with $X^{(n)}$ is given by $p_t^{(n)}=e^{t L^{(n)}}$,  $t>0$.    This and the fact that $L^{(n)}\bone_{\cK^{(n)}}=0$ imply the following lemma.

\begin{lemma}\label{lem:feller}
The semigroup $\{p_t^{(n)}\}_{t> 0}$  is strongly continuous and contractive on $\mathcal{B}_b(\cK^{(n)})$.  The generator of $\{p_t^{(n)}\}_{t> 0}$  is given by $(L^{(n)},\mathcal{B}_b(\cK^{(n)}))$.   Moreover,  we have $p_t^{(n)}\bone_{\cK^{(n)}}=1$ for any $t> 0$ and $\xi \in \cK^{(n)}$.
\end{lemma}

Define a c\`{a}dl\`{a}g process $Y^{(n)}=\{Y_{t}^{(n)}\}_{t \ge 0}$ on $\R^d$ by 
\[
Y_t^{(n)}=\cl{X_{t}^{(n)}},\quad t \ge 0.
\] 
In order to prove Theorem~\ref{thm:4} under the assumptions imposed there, we first obtain the weak convergence of $\{Y^{(n)}\}_{n=1}^\infty$.  We may assume that  $\bar{\xi} \in \cl{D}$  for any $\xi \in \cK^{(n)}$.  
Let $f \in C_{c,\text{Neu}}^2(\cl{D})$.  Since $\widetilde{\pi}_n f$ is a bounded function on $\cK^{(n)}$,  Lemma~\ref{lem:feller} implies that $\widetilde{\pi}_n f$ belongs to the domain of the generator associated with $\{p_{t}^{(n)}\}_{t>0}$.   The semigroup $\{p_t\}_{t>0}$ of the RBM $X$ is strongly continuous and contractive on $C_{\infty}(\cl{D})$.  Therefore,  under the assumption  that $ C_{c,\text{Neu}}^2(\cl{D})$ is  a core for $(\mathcal{L},\text{Dom}(\mathcal{L}))$,   \eqref{eq:cg'} and  \cite[Theorem~1.6.1]{EK} imply that  for any $t>0$ and $f \in C_{\infty}(\cl{D})$,
\begin{equation}\label{eq:cg''} 
\lim_{n \to \infty}\sup_{\xi \in \cK^{(n)}}\left| p_{t}^{(n)}(\widetilde{\pi}_n f)(\xi)- \widetilde{\pi}_n \left(p_tf\right)(\xi) \right|=0.
\end{equation}
From Lemma~\ref{lem:feller},  we have $P_{\xi}^{(n)}(X^{(n)}_t \in \cK^{(n)}\text{ for any $t \ge 0$})=1$ for any $\xi \in \cK^{(n)}$.
Since we have $\bar{\xi} \in \cl{D}$ for any $\xi \in \cK^{(n)}$,  the process $Y^{(n)}$ has sample paths in  $\mathcal{D}([0,\infty),\cl{D})$.   Here,   $\mathcal{D}([0,\infty),\cl{D})$ denotes the space   of right continuous functions on $[0,\infty)$ having left limits and  taking values in $\cl{D}$ that is equipped with the Skorohod topology.  
Let $x \in \cl{D}$,  and let $\xi^{(n)} \in \cK^{(n)}$ satisfy  $\lim_{n \to \infty}|\cl{\xi^{(n)}}-x|=0.$  Since $P_{\xi^{(n)}}^{(n)}(Y_0^{(n)}=\cl{\xi^{(n)}})=1$,  the initial distribution of $Y^{(n)}$ under $P_{\xi^{(n)}}^{(n)}$ converges weakly to the Dirac measure at $x$.  We can now use \eqref{eq:cg''} and \cite[Theorem~4.2.11]{EK} to obtain the following lemma.  

\begin{lemma}\label{lem:4}
Assume that all assumptions of Theorem~\ref{thm:4} are satisfied.  Let $x \in \cl{D}$,  and let $\xi^{(n)} \in \cK^{(n)}$($n=1,2,\ldots$) satisfy  $\lim_{n \to \infty}|\cl{\xi^{(n)}}-x|=0.$ Then, as $n \to \infty$,  the laws of  $\{(Y^{(n)},P_{\xi^{(n)}}^{(n)})\}_{n=1}^\infty$ converge weakly in $\mathcal{D}([0,\infty),\cl{D})$ to the law of the RBM on $\cl{D}$ starting from $x$.
\end{lemma}

Let $z \in \cl{D}$  and let $\xi \in \cK^{(n)}$.  Then,  we have 
\[
 \left|d_{\mathrm{H}}(\xi,z)-d_{\mathrm{H}}(\bar{\xi},z)\right|   \le d_{\mathrm{H}}(\bar{\xi},\xi).
\]
From the definition of the Hausdorff metric,  we have $d_{\mathrm{H}}(\bar{\xi},z)=|\bar{\xi}-z|$  and $d_{\mathrm{H}}(\bar{\xi},\xi)=\sup_{y \in \xi}|y-\bar{\xi}|\le \eps_n(\xi).$  Therefore,  we obtain that
\begin{equation}\label{eq:hd2}
 \left|d_{\mathrm{H}}(\xi,z)-|\bar{\xi}-z|\right|   \le \eps_n(\xi).
\end{equation}
This immediately implies the following lemma.

\begin{lemma}\label{lem:hd2}
Assume $\lim_{n \to \infty}\sup_{\xi \in \cK^{(n)}}\eps_n(\xi)=0.$  Let $x \in \cl{D}$,  and let $\xi^{(n)} \in \cK^{(n)}$($n=1,2,\ldots$) satisfy   $\lim_{n \to \infty}d_{\mathrm{H}}(\xi^{(n)},x)=0.$  Then,  we have $\lim_{n \to \infty}|\cl{\xi^{(n)}}-x|=0.$
\end{lemma}

Let $d_{\mathrm{S}}$ denote the Skorohod metric on $\mathcal{D}([0,\infty),\cl{D})$.  From (5.2) in \cite[Chapter~3]{EK},  we have for any $\omega,\omega' \in \mathcal{D}([0,\infty),\cl{D})$,
\begin{equation}\label{eq:skm}
d_{\mathrm{S}}(\omega,\omega')\le \int_{0}^\infty e^{-u}\sup_{t \in [0,u]}\left\{d_{\mathrm{H}}(\omega_t,\omega'_t) \wedge 1 \right\}\,du.
\end{equation}
We are now ready to prove Theorem~\ref{thm:4}.

\begin{proof}[Proof of Theorem~\ref{thm:4}]
Since $\lim_{n \to \infty}\sup_{\xi \in \cK^{(n)}}\eps_n(\xi)/\rho_n(\xi)
=0$,  we have \[\lim_{n \to \infty}\sup_{\xi \in \cK^{(n)}}\eps_n(\xi)=0.\]   Thus, Lemmas~\ref{lem:4}  and \ref{lem:hd2} imply that as $n \to \infty$, 
the laws of  $\{(Y^{(n)},P_{\xi^{(n)}}^{(n)})\}_{n=1}^\infty$ converge weakly in $\mathcal{D}([0,\infty),\mathcal{X})$ to  the law of the RBM $(X,P_x)$ on $\cl{D}$ starting from $x$.   
Let  $F \colon \mathcal{D}([0,\infty),\mathcal{X}) \to \R$ be a bounded Lipschitz continuous function.  The Portmanteau theorem implies that
\begin{equation}\label{eq:mr1}
\lim_{n \to \infty}E^{(n)}_{\xi^{(n)}}[F(Y^{(n)})]=E_{x}[F(X)].
\end{equation}
Here,  $E^{(n)}_{\xi^{(n)}}$ and $E_x$ denote the expectations under $P^{(n)}_{\xi^{(n)}}$ and $P_x$,  respectively. From \eqref{eq:hd2} and \eqref{eq:skm},   we have for any $n \in \N$,
\[
d_{\mathrm{S}}(X^{(n)},Y^{(n)})\le \int_{0}^\infty e^{-u} \sup_{\xi \in \cK^{(n)}}\eps_n(\xi)\,du=\sup_{\xi \in \cK^{(n)}}\eps_n(\xi),\quad P_{\xi^{(n)}}^{(n)}\text{-a.s.}
\]
Therefore, 
\begin{equation}\label{eq:mr2}
\lim_{n \to \infty}E^{(n)}_{\xi^{(n)}}[F(X^{(n)})-F(Y^{(n)})]=0.
\end{equation}
It follows from \eqref{eq:mr1} and \eqref{eq:mr2} that $\lim_{n \to \infty}E^{(n)}_{\xi^{(n)}}[F(X^{(n)})]=E_{x}[F(X)]$.
\end{proof}

\section{concluding remarks}

We conclude the present paper with mentioning a few future problems.

\begin{itemize}
 \item In order to obtain Lemmas~\ref{thm:1} and \ref{thm:2}, we needed to assume \eqref{eq:A} and \eqref{eq:B} to assure that $\rho(\xi)$ was sufficiently larger than $\eps(\xi)$ for any $\xi \in \cK$. In other words, the convergence results (Theorems~\ref{thm:3} and \ref{thm:4}) are proved under the condition that the range where the continuous time Markov chains $X^{(n)}$ move from each point $\xi$ has to be relatively large. As discussed in Remark~\ref{rem:srw}, however, in typical situations where the whole space is approximated by lattices the convergence is proven for Markov chains moving only nearest neighborhoods. So, a natural problem is how small $\rho(\xi)$ can be taken for the convergence of the processes.

\item The continuous parameter Markov chain $X^{(n)}$ has a holding time whose law is the exponential distribution with parameter, say $\lambda_n(\xi)$, when moving from $\xi\in \cK^{(n)}$. We can define another Markov process $Z^{(n)}$ by replacing the holding time of $X^{(n)}$ with the deterministic holding time with length $\lambda_n(\xi)$ at each $\xi$. We can also consider a continuous process $W^{(n)}$ on $\cl{D}$ by the geodesic interpolation of $\overline{Z^{(n)}}$. We expect similar convergences for $Z^{(n)}$ and $W^{(n)}$, but the proof does not seem straightforward.
\item In the paper~\cite{BIK} by Burago, Ivanov, and Kurylev, it is shown that the eigenvalues and eigenfunctions of the Laplacian on a closed Riemannian manifold $M$ are approximated by those of the graph Laplacians on partitions of $M$. 
 In the framework of the present paper, we can consider a similar problem since  the Neumann Laplacian $\mathcal{L}$ has a  discrete spectrum when $D$ is a smooth bounded domain. That is, as $n \to \infty$, we are concerned with whether eigenfunctions and eigenvalues of $L^{(n)}$ converge to those of  $\mathcal{L}$. A study by Ba\~{n}uelos and Pang~\cite{BP} also discussed approximating eigenvalues and eigenfunctions of  $\mathcal{L}$, where a weak convergence result to RBM obtained by Burdzy and Chen~\cite{BC3} was utilized. See \cite[Proposition~2.1]{BP} for details. The method of Ba\~{n}uelos and Pang~\cite{BP} may be helpful in solving this problem.

\end{itemize}

\begin{appendix}
\section{}

For $r>0$,  we define 
\[
B_{+}(0,r)=\{x \in B(0,r) \mid x_d>0\}.
\]
Recall that $\mathrm{B}(\cdot,\cdot)$ denotes the beta function and $\beta_d$ is defined as 
 \[\beta_d=\frac{2}{(d+1)\mathrm{B}(\frac{1}{2},\frac{d+1}{2})}.\]
\begin{lemma}\label{lem:uph}
For $r>0$, we have
\[
\frac{1}{r}\fint_{B_{+}(0,r)} y\,m(dy)=\beta_d e_d,
\]
where $e_d=(0,\ldots,0,1)$.  
\end{lemma}
\begin{proof}
If $d=1$,  we have $\mathrm{B}(\frac{1}{2},\frac{d+1}{2})=2$ and the conclusion immediately follows. Next,  we assume $d \ge 2.$ By symmetry, for any integer $i$ with $1\le i\le d-1$, we have
\[
\int_{B_{+}(0,r)} y_i\,m(dy)=0.\]
For $s \in \R$,  we let $
B^{(d),s}(0,r)=B(0,r) \cap \{x \in \R^d \mid x_d=s\}$.
 Let $\sigma$ denote the $(d-1)$-dimensional Hausdorff measure on $\R^d.$ Then,  we obtain that
\begin{align*}
\int_{B_{+}(0,r)} y_d\,m(dy)&=\int_{0}^r s \sigma(B^{(d),s}(0,r))\,ds =\omega_{d-1}\int_{0}^{r} s(r^2-s^2)^{(d-1)/2}\,ds\\
&=\omega_{d-1} \left[ -\frac{1}{d+1}(r^2-s^2)^{(d+1)/2}\right]_{0}^r=\frac{\omega_{d-1} r^{d+1}}{d+1}.
\end{align*}
Since $m(B_{+}(0,r))=\omega_d r^{d}/2$, 
\begin{equation}\label{eq:uph0}
\frac{1}{r}\fint_{B_{+}(0,r)} y_d\,m(dy)=\frac{2\omega_{d-1}}{(d+1)\omega_d }e_d.
\end{equation}
It is known that
\begin{equation}\label{eq:gamma}
\omega_d=\frac{\pi^{d/2}}{\mathrm{\Gamma} \left(\frac{d}{2}+1\right)},
\end{equation}
where $\mathrm{\Gamma}$ denotes the gamma function.  It is also known that
\begin{equation}\label{eq:beta2}
\textrm{B}(a,b)=\frac{\mathrm{\Gamma}(a)\mathrm{\Gamma}(b)}{\mathrm{\Gamma}(a+b)}
\end{equation}
for any $a>0$ and $b>0$.   We see from \eqref{eq:gamma} and \eqref{eq:beta2} that
\begin{align}
\frac{\omega_{d-1}}{\omega_d }=\frac{  \pi^{(d-1)/2} }{ \mathrm{\Gamma}\left(\frac{d-1}{2}+1\right)} \frac{\mathrm{\Gamma}\left(\frac{d}{2}+1\right)}{ \pi^{d/2}}&=\frac{\mathrm{\Gamma}(\frac{d+2}{2})}{\mathrm{\Gamma}(\frac{d+1}{2})\mathrm{\Gamma}(\frac{1}{2})}= \frac{1}{\mathrm{B}(\frac{1}{2},\frac{d+1}{2})}. \label{eq:uph1}
\end{align}
In the second equation,  we used the fact that $\mathrm{\Gamma}(1/2)=\pi^{1/2}$.   Substituting \eqref{eq:uph1} into \eqref{eq:uph0} completes the proof.
\end{proof}

\begin{lemma}\label{lem:uph2}
For $r>0$,  we have
\[
\fint_{B_{+}(0,r)} y\otimes y\,m(dy)=\frac{r^2}{d+2}I_d.
\]
\end{lemma}

\begin{proof}
Let $i,j$ be integers such that $1 \le i<j \le d$.   
Let 
\[B^{(i)}_{+}=\{ y^{(i)} \in \R^{d-1} \mid |y^{(i)}|<r,\,y_d \ge 0\},\]
where   $y^{(i)}=(y_1,\ldots,y_{i-1},y_{i+1},\ldots, y_d)$ and $|y^{(i)}|$ denotes the $(d-1)$-dimensional Euclidean  norm of $y^{(i)}$.  We have 
\begin{equation}\label{eq:uph2}
\int_{B_{+}(0,r)} y_i y_j\,m(dy)=\int_{B^{(i)}_{+}} \left( \int_{-\sqrt{r^2-|y^{(i)}|^2}}^{\sqrt{r^2-|y^{(i)}|^2}} y_i \,dy_i\right) y_j \,dy^{(i)}=0.
\end{equation}
Letting $i$ be an integer satisfying $1 \le  i \le d$,  we have  
\[
\int_{B_{+}(0,r)} y_i^2\,m(dy)=\frac{1}{2}\int_{B(r)} y_i^2\,m(dy).
\]
We use this and  \eqref{eq:a0} to obtain the following:
\begin{equation}\label{eq:uph3}
\fint_{B_{+}(0,r)} y_i^2\,m(dy)=\fint_{B(r)} y_i^2\,m(dy)=\frac{1}{d}\fint_{B(r)} |y|^2\,m(dy)=\frac{r^2}{d+2}.
\end{equation}
The conclusion follows from \eqref{eq:uph2} and \eqref{eq:uph3}.
\end{proof}

Using Lemma~\ref{lem:uph},  we also get the following lemma.
\begin{lemma}\label{lem:ub}
Let $D \subset \R^d$ be a  $C^{1,\alpha}$-domain for some $\alpha \in (0,1].$  
 Then, there exists $C>0$ depending only on  $D$ such that for any $x \in \partial D$ and $r \in (0,R\wedge 1),$
\[
 m\left( B_D (x,r)  \bigtriangleup  B_{+}(x,r) \right) \le Cr^{d+\alpha},
\]
where  $R$ is the positive constant in Definition~\ref{defn:Dc1a},   $B_D (x,r)=B(x,r)\cap D$,
\[
B_{+}(x,r)=\{y=(y',y_d) \in B(0,r) \text{ in $CS_x$} \mid y_d>0\},\]
and  $CS_x$ is an orthonormal coordinate system in \eqref{eq:Dcs}.   Moreover,   
there exists $C>0$ depending only on  $D$ such that for any $r \in (0,R\wedge 1),$
\begin{equation*}
\sup_{x \in \partial D}\left|\fint_{ B_D (x,r)} \frac{y-x}{r}\,m(dy) -\beta_d\nu(x) \right| \le Cr^\alpha.
\end{equation*}
Here,  $\nu(x)$ is the inward unit normal vector at $x$.
\end{lemma}
\begin{proof}
Let $x \in \partial D$ and $r\in (0,R  \wedge 1)$.   It then follows that
\begin{align} \label{eq:723}
&B_D (x,r)  \bigtriangleup  B_{+}(x,r)   \\
&\subset \left\{y=(y',y_d) \in B(0,r) \text{ in $CS_x$} \relmiddle| |y_d| \le  \sup_{y \in B(0,r)}|F_{x}(y')| \right\}. \notag
\end{align}
Here,  $F_x\colon \R^{d-1} \to \R$ is the $C^{1,\alpha}$-function in Definition~\ref{defn:Dc1a}.
The mean value theorem implies that  for any $y \in B(0,r)$ in $CS_x$,
 \begin{align}\label{eq:724}
|F_{x}(y')|&=|F_{x}(y')-F_{x}(0)| \\
&\le r \sup_{y' \in \R^{d-1},\,|y|\le r}|\nabla F_{x}(y')-\nabla F_{x}(0)| \notag  \\
&\le Cr^{\alpha+1}. \notag
 \end{align}
 Here,  $C$ is a positive constant in  Definition~\ref{defn:Dc1a}. 
From \eqref{eq:723} and \eqref{eq:724},  there exists $C'>0$ depending only on $D$ such that 
\begin{equation}\label{eq:725}
 m\left( B_D (x,r)  \bigtriangleup  B_{+}(x,r) \right) \le C'r^{d+\alpha}.
\end{equation}
It follows from Lemma~\ref{lem:uph} that
\begin{align*}
&\fint_{ B_{D}(x,r)} \frac{y-x}{r}\,m(dy) -\beta_d\nu(x) \notag \\
&=\fint_{ B_{D}(x,r)} \frac{y-x}{r}\,m(dy) -\fint_{ B_{+}(x,r)} \frac{y-x}{r}\,m(dy)  \\
&=\frac{1}{m(B_{D}(x,r))}\left(\int_{ B_{D}(x,r)} \frac{y-x}{r}\,m(dy)-\int_{ B_{+}(x,r)} \frac{y-x}{r}\,m(dy)  \right)\\
&\quad+\left(\frac{1}{m(B_{D}(x,r))}-\frac{1}{m(B_{+}(x,r))} \right)\int_{ B_{+}(x,r)} \frac{y-x}{r}\,m(dy).
\end{align*}
This implies that
\[
\left|\fint_{ B_{D}(x,r)} \frac{y-x}{r}\,m(dy) -\beta_d\nu(x)  \right| \le \frac{2m\left( B_D (x,r)  \bigtriangleup  B_{+}(x,r) \right) }{m(B_{D}(x,r))} .
\]
Applying \eqref{eq:D} and \eqref{eq:725} to the right-hand side of the above inequality,  we obtain the conclusion.
\end{proof}

\section{}
In this appendix, we prove Proposition~\ref{prop:core}.  
Henceforth,  we assume that $D$ is a bounded $C^{1,\alpha}$-domain for some $\alpha \in (0,1].$   For $k \in \{1,2\}$,  we denote by $H^{k}(D)$ the $k$-th order $L^2$-Sobolev space on $D$.   That is,  $H^{k}(D)$ is defined as
\begin{align*}
H^{k}(D)=\left\{f \in L^2(D,m)  \relmiddle|\parbox{0.54\textwidth}{For any $\alpha=(\alpha_1,\ldots,\alpha_d) \in (\mathbb{N}\cup \{0\})^d$ with $|\alpha|\le k$, the weak derivative $D^\alpha f$ exists and belongs to $L^2(D,m)$}\right\}.
\end{align*}
Here,  $|\alpha|=\alpha_1+\alpha_2+\cdots+\alpha_d$ and 
\[
D^\alpha f(x)=\frac{\partial^{|\alpha|} f}{\partial x_1^{\alpha_1} \partial x_2^{\alpha_2} \cdots \partial x_d^{\alpha_d}}(x),\quad x \in D.
\]
For $f\in H^1(D)$,  we define $\nabla f=(\partial f/\partial x_1,\ldots, \partial f/\partial x_d ).$  For $f,g \in H^1(D)$,  we define $\mathcal{E}(f,g)$ by \[\mathcal{E}(f,g)=\frac{1}{2}\int_{D}\langle \nabla f(x),\nabla g(x) \rangle\,m(dx).\]
For $\beta \ge 0$,  we let \[\mathcal{E}_\beta (f,g)=\mathcal{E}(f,g)+\beta \int_{D}f(x)g(x)\,m(dx).\]

Let $X=(\{X_t\}_{t \ge 0},\{P_x\}_{x \in \cl{D}})$ be the  RBM on $\cl{D}.$ As explained in Section~2,  $X$ can be constructed as a Feller process on $\cl{D}$.   Moreover,  $X$ is an $m$-symmetric diffusion process on $\cl{D}$,   and the Dirichlet form is identified with $(\mathcal{E},H^{1}(D))$   (see,  e.g., \cite[Theorem~2.2]{FT}).  Recall that $(\mathcal{L},\text{Dom}(\mathcal{L}))$ is  the generator of the Feller process $X$ on $\cl{D}$.  Let $\{G_\beta\}_{\beta>0 }$ denote the resolvent of $(\mathcal{L},\text{\rm Dom}(\mathcal{L}))$.  From the $m$-symmetry of $X$,  the resolvent  $\{G_\beta\}_{\beta>0 }$ is extended to a strongly continuous contraction resolvent on $L^q(\cl{D},m)$ for any $q \in [1,\infty)$.  Moreover,  for any $f \in L^2(D,m)$ and $\beta >0$, we have $G_\beta f \in H^1(D)$ and 
\begin{equation}\label{eq:weak}
\mathcal{E}_{\beta}(G_{\beta}f,g)=\int_{D}f(x)g(x)\,m(dx),\quad  g \in H^1(D).
\end{equation}
From \eqref{eq:weak}  and \cite[Corollary~8.11]{GT},  we obtain the following lemma.
\begin{lemma}\label{lem:intr}
For any $f \in C^{\infty}_c(\R^d)|_{\cl{D}}$ and $\beta >0$,  we have $G_\beta f \in C^\infty(D)$. 
\end{lemma}

For $k \in \{1,2\}$,  let $C^{k}(\cl{D})$ denote the set of functions in $C^{k}(D)$ all of whose derivatives (of order $\le k$)  extend to $\cl{D}$ continuously.
Since  $D$ is a bounded $C^{1,\alpha}$-domain for some $\alpha \in (0,1],$ the boundary $\partial D$  is $C^{1,\text{Dini}^2}.$ In particular,  $\partial D$ is $C^{1,\text{Dini}}$ (see \cite[Definitions~2.5]{DLK} for the definitions of the $C^{1,\text{Dini}}$  and $C^{1,\text{Dini}^2}$ boundaries).  Thus,  we conclude from \cite[Theorem~1.2]{DLK} that for any $\beta >0$ and $f \in L^q({D},m)$ with $q>d$,  there is an $m$-version of $G_\beta f$ that belongs to $C^1(\cl{D}).$
The version is also denoted by the same symbol.  The Gauss--Green formula and \eqref{eq:weak} imply that $G_\beta f $ satisfies the Neumann boundary condition on $\partial D$.  That is,  $\langle \nabla G_\beta f,\nu \rangle=0 $ on $\partial D.$  From this fact and Lemma~\ref{lem:intr},  we arrive at the following lemma.

\begin{lemma}\label{lem:cls}
For any $f \in C^{\infty}_c(\R^d)|_{\cl{D}}$ and $\beta >0$,  we have $G_\beta f \in C^\infty(D) \cap C^1(\cl{D})$.  Furthermore, $
\beta G_\beta f-(\Delta/2) G_\beta f=f$ on $D$  and   $\langle \nabla G_\beta f,\nu  \rangle=0 $ on  $ \partial D.$
\end{lemma}

\cite[Theorems~2.2.2.5 and 3.2.1.3]{GP} lead us to the following lemma.

\begin{lemma}\label{lem:h2}
If  $D$ is a bounded $C^{1,1}$-domain or a bounded convex $C^{1,\alpha}$-domain for some $\alpha \in (0,1]$,  then $G_\beta f \in H^2(D)$ for any $f \in L^2(D,m)$ and $\beta >0$. 
\end{lemma}

Any Lipschitz continuous function $f\colon \cl{D} \to \R$ is  of  Dini mean oscillation; see  \cite[Definition~2.7]{DLK} for the definition.  In particular,  any $f \in C^{\infty}_c(\R^d)|_{\cl{D}}$ is of  Dini mean oscillation.   
Hence,  we use Lemmas~\ref{lem:cls}, and \ref{lem:h2}  and \cite[Theorem~1.5]{DLK} to  obtain the following result.

\begin{lemma}\label{lem:core0}
Assume that $D$ is a bounded $C^{1,1}$-domain or a bounded convex $C^{1,\alpha}$-domain for some $\alpha \in (0,1]$.  Then,  for any $f \in C^{\infty}_c(\R^d)|_{\cl{D}}$ and $\beta >0$,  we have $G_\beta f \in C^2(\cl{D})$ and   $\langle \nabla G_\beta f,\nu  \rangle=0 $ on $ \partial D$.
\end{lemma}

We say that a domain $D' \subset  \R^d$ is {\it quasiconvex}  if there exists $C \in (0,\infty)$ such that any pair of points $x,y \in D'$ can be joined by a curve $\gamma$ in $D'$ satisfying $\text{length}(\gamma) \le C|x-y|.$  Here,  $\text{length}(\gamma)$ denotes the length of $\gamma$.
As stated in \cite[Example~3.39]{BB},  any bounded Lipschitz domain is quasiconvex.  
In particular,   any bounded $C^{1,\alpha}$-domain ($\alpha \in (0,1]$) is quasiconvex.

\begin{proof}[Proof of Proposition~\ref{prop:core}]
Let $\beta>0.$  Since $ C^{\infty}_c(\R^d)|_{\cl{D}}$ is a dense subspace of $C(\cl{D})$ with respect to the sup-norm,  so is $G_\beta(C^{\infty}_c(\R^d)|_{\cl{D}})$.  
This  and  \cite[Proposition~1.3.1]{EK} show that $G_\beta(C^{\infty}_c(\R^d)|_{\cl{D}})$ is a core for $(\mathcal{L},\text{\rm Dom}(\mathcal{L}))$.   As explained in Section~2,   applying the It\^{o} formula to \eqref{eq:skorohod},  we see that $ C_{c,\text{Neu}}^2(\cl{D})$ is a subspace of $\text{Dom}(\mathcal{L})$.  Therefore,  the conclusion follows if we show that $G_\beta f \in  C_{c,\text{Neu}}^2(\cl{D})$ for any $f \in C^{\infty}_c(\R^d)|_{\cl{D}}$.  Let $f \in C^{\infty}_c(\R^d)|_{\cl{D}}$.   Since $D$ is quasiconvex,  Lemma~\ref{lem:core0} and 
the Whitney extension theorem~\cite[Theorem]{W} imply that $G_\beta f$ is extended to a $C^2$-function on $\R^d$   and   $\langle \nabla G_\beta f,\nu  \rangle=0 $ on $ \partial D$.   Consequently,  we have $G_\beta f \in  C_{c,\text{Neu}}^2(\cl{D})$.
\end{proof}

\end{appendix}

\begin{ack}
The authors thank the anonymous referees for their helpful comments that improved the quality of the paper.
\end{ack}

\bibliographystyle{amsplain}

\end{document}